\documentclass[12pts]{amsart}

\usepackage{amscd,latexsym,amsthm,amsfonts,amssymb,amsmath,amsxtra}
\usepackage[mathscr]{eucal}
\newcommand{\BA}{{\mathbb {A}}}

\newcommand{\BC}{{\mathbb {C}}}

\newcommand{\BQ}{{\mathbb {Q}}}
\newcommand{\BR}{{\mathbb {R}}}

\newcommand{\RX}{{\mathrm {X}}}
\newcommand{\RY}{{\mathrm {Y}}}
\newcommand{\RZ}{{\mathrm {Z}}}

\newcommand{\diag}{{\mathrm{diag}}}

\newcommand{\GL}{{\mathrm{GL}}}

\newcommand{\I}{{\mathrm{I}}}

\renewcommand{\Re}{{\mathrm{Re}}}

\newcommand{\SL}{{\mathrm{SL}}}

\newcommand{\Sp}{{\mathrm{Sp}}}

\newtheorem{thm}{Theorem}[section]

\newtheorem{lem}[thm]{Lemma}
\newtheorem{prop}[thm]{Proposition}

\newtheorem{rmk}[thm]{Remark}

\begin{document}
\renewcommand{\theequation}{\arabic{equation}}
\numberwithin{equation}{section}

\title{Poles of Eisenstein series on general linear groups, induced from two Speh representations}

\author{David Ginzburg}

\address{School of Mathematical Sciences, Sackler Faculty of Exact Sciences, Tel-Aviv University, Israel
69978} \email{ginzburg@tauex.tau.ac.il}

\thanks{This research was supported by the ISRAEL SCIENCE FOUNDATION
	(grant No. 461/18).}

\author{David Soudry}
\address{School of Mathematical Sciences, Sackler Faculty of Exact Sciences, Tel-Aviv University, Israel
69978} \email{soudry@tauex.tau.ac.il}





\maketitle

\begin{abstract}
We determine the poles of the Eisenstein series on $\GL_{n(m_1+m_2)}(\BA)$, induced from two Speh representations, $\Delta(\tau,m_1)|\cdot|^s\times \Delta(\tau,m_2)|\cdot|^{-s}$, $Re(s)\geq 0$, where $\tau$ is an irreducible, unitary, cuspidal, automorphic representation of $\GL_n(\BA)$. These poles are simple and occur at $s=\frac{m_1+m_2}{4}-\frac{i}{2}, \ \ 0\leq i\leq min(m_1,m_2)-1$. Our methods also show that, when $m_1=m_2$, the above Eisenstein series vanish at $s=0$.
\end{abstract}

\section{Introduction}

Let $\tau$ be an irreducible, unitary, cuspidal, automorphic representation of $\GL_n(\BA)$, where $\BA$ is the ring of adeles of a numbr field $F$. Let $m_1, m_2$ be two positive integers. In this paper, we consider Eisenstein series on $\GL_{n(m_1+m_2)}(\BA)$, attached to the parabolic induction
$$
\rho_{\Delta(\tau,(m_1,m_2)),s}=\Delta(\tau,m_1)|\cdot|^s\times \Delta(\tau,m_2)|\cdot|^{-s},
$$
where $\Delta(\tau,r)$ denotes the Speh representation of $\GL_{nr}(\BA)$, corresponding to $\tau$. This is the irreducible, automorphic representation of $\GL_{nr}(\BA)$ generated by the multi-residues of Eisenstein series attached to $\tau|\cdot|^{\zeta_1}\times\cdots\times \tau|\cdot|^{\zeta_r}$, at the point $(\frac{r-1}{2},\frac{r-3}{2},...,\frac{1-r}{2})$. Such representations generate the discrete, non-cuspidal spectrum of $L^2(\GL_{nr}(F)\backslash \GL_{nr}(\BA))$. See the paper of Moeglin and Waldspurger \cite{MW89}.
The following theorem is stated by Z.Zhang, in \cite{Zh23}, for $F=\BQ$.
\begin{thm}\label{thm A}
	The list of poles of the Eisenstein series attached to $\rho_{\Delta(\tau,(m_1,m_2)),s}$, for $Re(s)\geq 0$, is
	$$
	s(i)=\frac{m_1+m_2}{4}-\frac{i}{2}, \ \ 0\leq i\leq min(m_1,m_2)-1.
	$$
	All these poles are simple.
\end{thm}
We remark that the points $s(i)$ are exactly the positive numbers $s$, such that the segments $[\frac{1-m_2}{2}-s,\frac{m_2-1}{2}-s]$ and $[\frac{1-m_1}{2}+s,\frac{m_1-1}{2}+s]$ are linked, in the sense of Zelevinsky, \cite{Ze80}, p. 184. Thus, the local analog of Theorem \ref{thm A}, for a local non-archimedean field $K$ and an irreducible, supercuspidal representation $\sigma$ of $\GL_n(K)$, is the fact that, for $Re(s)\geq 0$, the parabolic induction $\Delta(\sigma,m_1)|\cdot|^s\times \Delta(\sigma,m_2)|\cdot|^{-s}$ is irreducible iff $s=s(i)$, $0\leq i\leq min(m_1,m_2)-1$, where $\Delta(\sigma,r)$ is the unique irreducible quotient of the parabolic induction $\sigma|\cdot|^{\frac{r-1}{2}}\times\sigma|\cdot|^{\frac{r-3}{2}}\times\cdots\times\sigma|\cdot|^{\frac{1-r}{2}}$. See Theorem 4.2 in \cite{Ze80}.

In the case $n=1$, Theorem \ref{thm A} was proved by Hanzer and Muic in \cite{HM15}. The proof in \cite{Zh23} follows the combinatoric method of \cite{HM15}, but is problematic, as there is not enough attention paid to poles which might arise from local intertwining operators. 

In this paper, we present a completely different proof, valid for any number field $F$. Let $f_{\Delta(\tau,(m_1,m_2)),s}$ be a smooth holomorphic section of $\rho_{\Delta(\tau,(m_1,m_2)),s}$, and denote the corresponding Eisenstein series by $E(f_{\Delta(\tau,(m_1,m_2)),s})$. We first show in Theorem \ref{thm 3.1}, that $s=\frac{m_1+m_2}{4}$ is a simple pole of our Eisenstein series, and that it is its right-most pole. The proof is straightforward, by considering the constant term of $E(f_{\Delta(\tau,(m_1,m_2)),s})$, along the unipotent radical of the parabolic subgroup corresponding to the partition $(n^{m_1+m_2})$. Next, we show that, for $1\leq i\leq min(m_1,m_2)-1$, $s(i)$ is a pole of $E(f_{\Delta(\tau,(m_1,m_2)),s})$ (for some section). We prove this in Theorem \ref{thm 4.1}, by applying to $E(f_{\Delta(\tau,(m_1,m_2)),s})$ a Fourier coefficient analogous to the Bernstein-Zelevinsky derivative (of order $2n$), stabilized by $\GL_{n(m_1+m_2-2)}(\BA)$. We denote the application of this Fourier coefficient by $\mathcal{D}_{\psi,2n}$. For an automorphic representation $\pi$ of $\GL_{n(m_1+m_2)}(\BA)$, we view $\mathcal{D}_{\psi,2n}(\pi)$ (by restriction) as an automorphic representation of the stabilizer $\GL_{n(m_1+m_2-2)}(\BA)$. ($\psi$ is a fixed nontrivial character of $F\backslash \BA$). It turns out that $\mathcal{D}_{\psi,2n}(E(f_{\Delta(\tau,(m_1,m_2)),s}))$ is, up to a power of $|\det\cdot|$, an Eisenstein series corresponding to $\rho_{\Delta(\tau,(m_1-1,m_2-1)),s}$. This is a special case of Theorem \ref{thm 2.2}. This property is a global analog of the computation of derivatives applied to parabolic induction of segment representations. See Remark \ref{rmk 2.2}. Now, we apply induction on $m_1+m_2$ (and the results of Sec. 3.3, 3.4). We note that we have used a similar idea in \cite{GS22}, where we applied a descent operation to an Eisenstein series on $\Sp_{2nm}(\BA)$, induced from $\Delta(\tau,m)$ and the Siegel parabolic subgroup. It gave an Eisenstein series on $\Sp_{2n(m-1)}(\BA)$, induced from $\Delta(\tau,m-1)$. Thus, poles for the last case must come from poles of the former case. In this paper, $\mathcal{D}_{\psi,2n}$ provides a descent operation, taking an Eisenstein series attached to $\rho_{\Delta(\tau,(m_1,m_2)),s}$ to one attached to $\rho_{\Delta(\tau,(m_1-1,m_2-1)),s}$. We remark that Zhang considered the derivative-type coefficient above, as well as others, in another paper, \cite{Zh22}, and proved the descent property above. However, he did not consider poles that might arise from the section of the new Eisenstein series, nor the problem of how general this section is. These issues are taken care of in Sec. 3.3, 3.4. In \cite{Zh22}, Zhang applied the derivative-type coefficient above to the residue at the simple pole $s(i)$ (whose existence he knew from \cite{Zh23}), in order to find the wave front set of the automorphic representation $\pi_i$, generated by the residues at $s(i)$ of the Eisenstein series above. In this paper, we use $\mathcal{D}_{\psi,2n}$ in order to prove by induction, that $s(i)$, $1\leq min(m_1,m_2)-1$, is a pole of the Eisenstein series above. 

Denote $m=m_1+m_2$. What we explained so far, proves that the points $s(i)$, $0\leq i\leq min(m_1,m_2)-1$, are poles of the Eisenstein series attached to $\rho_{\Delta(\tau,(m_1,m_2)),s}$, for $Re(s)\geq 0$, that $s(0)$ is a simple pole, and that it is the right-most pole. The next part of the proof of Theorem \ref{thm A} consists of showing that all the poles $s(i)$ are simple, and that there are no other poles when $Re(s)\geq 0$. This is done in Sec. 5.2. The idea here is to consider automorphic representations $\pi$ of $\GL_{mn}(\BA)$, generated by the leading terms of the Laurent expansions of our Eisenstein series, at a pole $s_0$, such that $\mathcal{D}_{\psi,2n}(\pi)=0$. We prove in Prop. \ref{prop 4.3}, that the set of maximal nilpotent orbits, whose corresponding Fourier coefficients are nontrivial on $\pi$, $\mathcal{O}(\pi)$, is a singleton and corresponds to the partition $(n^m)$. In such a case, we find from Theorem \ref{thm 1.6}, an explicit relation between two Fourier coefficients applied to $\pi$. The first is the Fourier coefficient attached to the partition $(n^m)$, $\mathcal{F}_{\psi_{(n^m)}}$. The second is the Fourier coefficient, $\mathcal{Z}_{\psi, n,m}$, obtained by first taking a constant term along the unipotent radical of the parabolic subgroup $P_{n^m}$ corresponding to the partition $(n^m)$, and then applying the Whittaker coefficient, with respect to $\psi$, along the standard maximal unipotent subgroup of the Levi part. This explicit passage between the two Fourier coefficients is done by the method of exchanging roots. Next, we use the following invariance property from Theorem \ref{thm 1.7},
$$
\mathcal{F}_{\psi_{(n^m)}}(\varphi)(\delta_n(a)g)=\mathcal{F}_{\psi_{(n^m)}}(\varphi)(g),
$$
where $\varphi\in \pi$, $a\in \SL_m(\BA)$, $\delta_n(a)=\diag(a,...,a)$ ($n$ times). Transferring this invariance to $\mathcal{Z}_{\psi, n,m}(\varphi)$ forces $s_0$ to be $\frac{m}{4}$. It is interesting that these ideas can be used to show
\begin{thm}\label{thm B}
	Assume that $m_1=m_2=k$. Then the value of any Eisenstein series attached to $\rho_{\Delta(\tau,(k,k)),s}$, at $s=0$, is zero.
\end{thm}
This is Theorem \ref{thm 5.1}. The case $k=1$ follows from a result of Keys and Shahidi (Prop. 6.3, \cite{KS88}). Finally, let $\pi_i$ denote the representation of $\GL_{mn}(\BA)$ by right translations, in the space generated by the residues $Res_{s=\frac{m}{4}-\frac{i}{2}}E(f_{\Delta(\tau,(m_1,m_2)),s})$, $0\leq min(m_1,m_2)-1$. Now that we know that $\mathcal{D}_{\psi,2n}(\pi_i)\neq 0$, for $1\leq min(m_1,m_2)-1$, and the results of Sec. 3.3, 3.4, we use induction on $m$ to show
\begin{thm}\label{thm C}
Let $0\leq i\leq min(m_1,m_2)-1$. Then
$$
\mathcal{O}(\pi_i)=((2n)^i,n^{m-2i}).
$$
\end{thm}
Here, the proof for $i=0$ (and any $m$) is straightforward.	The proof by induction on $m$, for $1\leq min(m_1,m_2)-1$, is similar to \cite{Zh22}, once we have the results of Sec. 3.3, 3.4. Again, we used a similar idea in \cite{GS22}.
\vspace{0.5cm}

\noindent {\bf Some notation.} Let $\underline{k}=(k_1,...,k_r)$ be partition $\ell$, i.e. $k_1,...,k_r$ are positive integers, whose sum is $\ell$. Let $F$ be a field. We denote by $P_{\underline{k}}(F)$ the standard parabolic subgroup of $\GL_\ell(F)$, whose Levi part is
$$
M_{\underline{k}}(F)=\{\diag(g_1,...,g_r):\ g_1\in \GL_{k_1}(F),...,g_r\in \GL_{k_r}(F)\}.
$$
We denote the unipotent radical by $V_{\underline{k}}(F)$. Oftentimes, when $\underline{k}$ is explicit, we omit the parentheses and denote $P_{k_1,...,k_r}$, $M_{k_1,...,k_r}$, $V_{k_1,...,k_r}$.\\
We denote, for positive integers $k, r$, by $(k^r)$, the partition $(k,...,k)$, $k$ appearing $r$ times. Thus, $V_{1^r}(F)$ is the standard maximal unipotent subgroup of $\GL_r(F)$. We usually also denote
$$
Z_r(F)=V_{1^r}(F).
$$
Let $F$ be a number field and $\BA$ its ring of adeles. Throughout the paper, $\psi$ denotes a fixed nontrivial character of $F\backslash \BA$.\\
For $\underline{k}$ as above, let $\sigma_i$ be an automorphic representation of $\GL_{k_i}(\BA)$, $1\leq i\leq r$. Extend $\sigma_1\otimes\cdots\otimes\sigma_r$ from $M_{\underline{k}}(\BA)$ to $P_{\underline{k}}(\BA)$, by the trivial action of $V_{\underline{k}}(\BA)$. We denote the induction of this representation from $P_{\underline{k}}(\BA)$ to $\GL_\ell(\BA)$ by
$$
\sigma_1\times\cdots\times\sigma_r.
$$
We use similar notation over a local field. Finally, for an algebraic group $G$ over a number field $F$, we denote $$
[G]=G(F)\backslash G(\BA).
$$

\section{Nilpotent orbits and Fourier coefficients}

In this section, we recall the definition of Fourier coefficients with respect to characters of unipotent groups, attached to nilpotent orbits. We restrict to general linear groups. We follow \cite{GRS03}. We then study two examples: Fourier coefficients attached to the partition $(k,1^{\ell-k})$ of $\ell$ and Fourier coefficients attached to the partition $(k^r)$ of $kr$.

\subsection{Unipotent groups and characters attached to partitions}  Let $F$ be a field. Let $\ell$ be a positive integer and consider the group $\GL_\ell(F)$ acting by conjugation on its Lie algebra, which we identify with the $\ell\times \ell$ matrices, $M_\ell(F)$. Using the Jordan form of a matrix, we know that nilpotent orbits in $M_\ell(F)$ are determined by descending partitions of $\ell$, $\underline{p}=(p_1,...,p_r)$, $p_1\geq\dots\geq p_r\geq 1$. We will associate to $\underline{p}$ a one parameter subgroup $h_{\underline{p}}(t)$, inside the diagonal subgroup of $\GL_\ell(F)$, and an upper unipotent subgroup $N_{\underline{p}}(F)\subset GL_\ell(F)$, which is a unipotent radical of a standard parabolic subgroup.

Start with a partition of $\ell$,
\begin{equation}\label{1.1}
\underline{p}=((2k)^{m_k},(2k-1)^{n_k},(2k-2)^{m_{k-1}-m_k},(2k-3)^{n_{k-1}-n_k},...,2^{m_1-m_2},1^{n_1-n_2}),
\end{equation} 
where $m_1\geq\dots\geq m_k\geq 0$, $n_1\geq\dots\geq n_k\geq 0$, $m_k+n_k\geq 1$. In \eqref{1.1}, we mean that $2k$ is repeated $m_k$ times, $2k-1$ is repeated $n_k$ times etc. Note that $\ell=2(m_1+m_2+\cdots+m_k)+2(n_2+n_3+\cdots+n_k)+n_1$. We describe now how to form the torus element $h_{\underline{p}}(t)$. Consider first a partition of the form $(a^b)$ (of $ab$, for positive integers $a,b$). Denote, for $t\in F^*$,
\begin{equation}\label{1.2}
h_{a^b}(t)=\begin{pmatrix}t^{a-1}I_b\\&t^{a-3}I_b\\&&\ddots\\&&&t^{1-a}I_b\end{pmatrix}.
\end{equation}
Consider the diagonal matrix 
\begin{equation}\label{1.3}
\diag(h_{(2k)^{m_k}}(t),h_{(2k-1)^{n_k}}(t),h_{(2k-2)^{m_{k-1}-m_k}}(t),...,h_{2^{m-1-m_2}}(t),h_{1^{n_1-n_2}}(t)).
\end{equation}
The matrix $h_{\underline{p}}(t)$ is obtained by permuting the powers of $t$ in \eqref{1.3}, so that they are in descending order. Thus,
\begin{equation}\label{1.3.1}
h_{\underline{p}}(t)=\diag(...,t^{2i-1}I_{m_i},t^{2i-2}I_{n_i},...,tI_{m_1},I_{n_1},t^{-1}I_{m_1},...,t^{2-2i}I_{n_i},t^{1-2i}I_{m_i},...).
\end{equation}
Here, $2\leq i\leq k$. Consider the action by conjugation of $h_{\underline{p}}(t)$ on $M_\ell(F)$. Let $\{e_{i,j}: 1\leq i,j\leq\ell\}$ be the standard basis of $M_\ell(F)$. All its elements are eigenvectors for the action of $h_{\underline{p}}(t)$. For $1\leq i, j\leq \ell$, there is an integer $r_{i,j}$ that can be determined by $\underline{p}$, such that
$$
h_{\underline{p}}(t)e_{i,j}h_{\underline{p}}(t)^{-1}=t^{r_{i,j}}e_{i,j}.
$$
Let $N_{\underline{p}}(F)$ be the subgroup of $\GL_\ell(F)$, generated by the matrices $I_\ell+xe_{i,j}$, $x\in F$, $r_{i,j}\geq 1$. From \eqref{1.3.1}, this is the unipotent radical of the standard parabolic subgroup $P_{\underline{p}'}(F)$, corresponding to the partition
\begin{equation}\label{1.4}
\underline{p}'=(m_k,n_k,...,m_2,n_2,m_1,n_1,m_1,n_2,m_2,...,n_k,m_k).
\end{equation}
Thus, by our notation, $N_{\underline{p}}=V_{\underline{p}'}$. The elements of the Levi part $M_{\underline{p}'}(F)$ are
\begin{equation}\label{1.5}
\diag(g_1,h_1,...,g_{k-1},h_{k-1},g_k,h_k,g'_k,h'_{k-1},g'_{k-1},...,h'_1,g'_1),
\end{equation}
where, for $1\leq i\leq k$, $g_i, g'_i\in \GL_{m_{k-i+1}}(F)$, $h_k\in \GL_{n_1}(F)$, and, for $1\leq j\leq k-1$, $h_j, h'_j\in \GL_{n_{k-j+1}}(F)$.
If some $m_{k-j+1}=0$, then we drop the blocks $g_j, g'_j$, and similarly when $n_{k-j+1}=0$. Let $N^2_{\underline{p}}(F)$ be the subgroup of $N_{\underline{p}}(F)$, generated by the matrices $I_\ell+xe_{i,j}$, $x\in F$, $r_{i,j}\geq 2$. This is a normal subgroup of $N_{\underline{p}}(F)$. Note that $N^2_{\underline{p}}(F)=N_{\underline{p}}(F)$ iff all $m_i$ are zero, or all $n_j$ are zero. We have
\begin{equation}\label{1.6}
i_{\underline{p}}:N_{\underline{p}}(F)/N^2_{\underline{p}}(F)\cong \bigoplus_{i=1}^{k-1}(M_{m_{k-i+1}\times n_{k-i+1}}(F)\oplus
 M_{n_{k-i+1}\times m_{k-i}}(F))
 \end{equation}
 $$
 \oplus (M_{m_1\times n_1}(F)\oplus M_{n_1\times m_1}(F))\oplus (\bigoplus_{i=1}^{k-1}(M_{m_i\times n_{i+1}}(F)\oplus M_{n_{i+1}\times m_{i+1}}(F))),
$$
and  
\begin{equation}\label{1.7}
j_{\underline{p}}: N^2_{\underline{p}}(F)/[N^2_{\underline{p}}(F),N^2_{\underline{p}}(F)]\cong \bigoplus_{i=1}^{k-1}(M_{m_{k-i+1}\times m_{k-i}}(F)\oplus
M_{n_{k-i+1}\times n_{k-i}}(F))
\end{equation}
$$
\oplus M_{m_1\times m_1}(F)\oplus (\bigoplus_{i=1}^{k-1}(M_{n_i\times n_{i+1}}(F)\oplus M_{m_i\times m_{i+1}}(F))).
$$
The isomorphisms above can be described as follows. Let $x\in N_{\underline{p}}(F)$. Think of $x$ as a matrix of blocks following \eqref{1.4}. Then in \eqref{1.6}, 
$$
i_{\underline{p}}(xN^2_{\underline{p}}(F))= (x_{1,2},x_{2,3},...,x_{4k-2,4k-1}), $$
where $x_{1,2}\in M_{m_k\times n_k}(F)$, $x_{2,3}\in M_{n_k\times m_{k-1}}(F)$, and so on. Similarly, if $x\in N^2_{\underline{p}}(F)$, then  $$
j_{\underline{p}}(x[N^2_{\underline{p}}(F),N^2_{\underline{p}}(F)])=  (x_{1,3},x_{2,5},...,x_{4k-3,4k-1}),
$$
where $x_{1,3}\in M_{m_k\times m_{k-1}}(F)$, $x_{2,4}\in M_{n_k\times n_{k-1}}(F)$, and so on. A space $M_{a\times b}(F)$ which appears in the sums above, such that $ab=0$, should be dropped, and the corresponding coordinate should also be droped from the last descriptions of $i_{\underline{p}}$, or $j_{\underline{p}}$.

Let us describe the action by conjugation of the Levi part $M_{\underline{p'}}(F)$ (corresponding to the partition \eqref{1.4}) on the quotients \eqref{1.6}, \eqref{1.7}. Take the element \eqref{1.5}. Then in \eqref{1.6}, $(x,y)\in M_{m_{k-i+1}\times n_{k-i+1}}(F)\oplus
M_{n_{k-i+1}\times m_{k-i}}(F)$, $1\leq i\leq k-1$, is mapped by conjugation to 
$(g_ixh_i^{-1},h_iyg_{i+1}^{-1})$; an element $(x,y)\in M_{m_1\times n_1}(F)\oplus M_{n_1\times m_1}(F)$ is mapped to $(g_kxh_k^{-1},h_ky(g'_k)^{-1})$; an element $(x,y)\in M_{m_i\times n_{i+1}}(F)\oplus M_{n_{i+1}\times m_{i+1}}(F)$, $1\leq i\leq k-1$, is mapped to
 $(g'_{k-i+1}x(h'_{k-i})^{-1},h'_{k-i}y(g'_{k-i})^{-1})$. Similarly, in \eqref{1.7}, an element
  $(x,y)\in M_{m_{k-i+1}\times m_{k-i}}(F)\oplus
M_{n_{k-i+1}\times n_{k-i}}(F)$ is mapped to $(g_ixg_{i+1}^{-1},h_iyh_{i+1}^{-1})$; an element $x\in M_{m_1\times m_1}(F)$ is mapped to $g_kx(g'_k)^{-1}$; an element
 $(x,y)\in M_{n_i\times n_{i+1}}(F)\oplus M_{m_i\times m_{i+1}}(F)$ is mapped to\\ $(h'_{k-i+1}x(h'_{k-i})^{-1},g'_{k-i+1}y(g'_{k-i})^{-1})$. Here $h'_k=h_k$.
Again, in case of a summand $M_{a\times b}(F)$ which appears in \eqref{1.6}, or \eqref{1.7}, such that $ab=0$, the description of the conjugation above is modified appropriately.

Define the following homomorphism $b_{\underline{p}}: N^2_{\underline{p}}(F)\rightarrow F$. 
Since this homomorphism must be trivial on $[N^2_{\underline{p}}(F),N^2_{\underline{p}}(F)]$, we will use the isomorphism \eqref{1.7} and define our homomorphism on each summand in \eqref{1.7}. We keep denoting this homomorphism by $b_{\underline{p}}$. Let $1\leq i\leq k-1$ and $(x_{k-i},y_{k-i})\in M_{m_{k-i+1}\times m_{k-i}}(\BA)\oplus
M_{n_{k-i+1}\times n_{k-i}}(\BA)$. Recall that $m_{k-i+1}\leq m_{k-i}$, $n_{k-i+1}\leq n_{k-i}$. Then
\begin{equation}\label{1.8}
b_{\underline{p}}(x_{k-i},y_{k-i})=tr(x_{k-i}\begin{pmatrix}I_{m_{k-i+1}}\\0\end{pmatrix})+tr(y_{k-i}\begin{pmatrix}I_{n_{k-i+1}}\\0\end{pmatrix}).
\end{equation}
For $x_0\in M_{m_1\times m_1}(\BA)$,
\begin{equation}\label{1.9}
b_{\underline{p}}(x_0)=tr(x_0).
\end{equation}For $1\leq i\leq k-1$ and $(y'_i,x'_i)\in M_{n_i\times n_{i+1}}(\BA)\oplus M_{m_i\times m_{i+1}}(\BA)$,
\begin{equation}\label{1.10}
b_{\underline{p}}(y'_i,x'_i)=tr(y'_i(I_{n_{i+1}},0))+tr(x'_i(I_{m_{i+1}},0)).
\end{equation}
As an example, let $\underline{p}=(4,2^2,1)$, i.e. $k=2$, $m_2=1, n_2=0,m_1=2,n_1=1$. Then
$$
b_{\underline{p}}(\begin{pmatrix}1&x_1&\ast&\ast&\ast\\&I_2&0&x_0&\ast\\&&1&0&\ast\\&&&I_2&x'_1\\&&&&1\end{pmatrix})=x_1\begin{pmatrix}1\\0\end{pmatrix}+tr(x_0)+tr(x'_1(1,0)).
$$

There is a unique matrix $\alpha_{\underline{p}}\in M_{\ell}(F)$, such that
$$
h_{\underline{p}}(t)\alpha_{\underline{p}}h_{\underline{p}}(t)^{-1}=t^{-2}\alpha_{\underline{p}},
$$
and, for $x\in N^2_{\underline{p}}(F)$, 
\begin{equation}\label{1.11}
b_{\underline{p}}(x)=tr (\alpha_{\underline{p}}log(x))=tr(\alpha_{\underline{p}}(x-I_\ell)).
\end{equation}
Note that $\alpha_{\underline{p}}$ is lower nilpotent. In the example above,
$$
\alpha_{\underline{p}}=\begin{pmatrix}0\\ \begin{pmatrix}1\\0\end{pmatrix}&0_{2\times 2}\\&&0\\&I_2&&0_{2\times 2}\\&&&(1,0)&0\end{pmatrix}.
$$

The stabilizer $R_{\underline{p}}(F)$ of $b_{\underline{p}}$ inside $M_{\underline{p}'}(F)$ is the centralizer of $\alpha_{\underline{p}}$ inside $M_{\underline{p}'}(F)$. It is isomorphic to 
\begin{equation}\label{1.12}
\prod_{i=0}^{k-1}(\GL_{m_{k-i}-m_{k-i+1}}(F)\times \GL_{n_{k-i}-n_{k-i+1}}(F)).
\end{equation}
Here, $m_{k+1}=n_{k+1}=0$. Also, if in \eqref{1.11}, in one of the factors $\GL_a(F)$, $a=0$, then we drop this factor. This can be seen by using the description above of the action of $M_{\underline{p'}}(F)$ on the quotient \eqref{1.7}. The isomorphism in \eqref{1.12} is such that each $\GL$-factor in \eqref{1.12} is embedded diagonally in $M_{\underline{p}'}(F)$, as follows. Let $g\in \GL_{m_{k-i}-m_{k-i+1}}(F)$. We describe the image of $g$ in each block of $M_{\underline{p}'}(F)$. Each block corresponding to $n_j$ in \eqref{1.4} is $I_{n_j}$. For $k-i+1\leq j\leq k$, the block corresponding to $m_j$ in \eqref{1.4} is $I_{m_j}$. For $1\leq j\leq k-i$, the block corresponding to $m_j$ has the form $\diag(I_{m_{k-i+1}},g,I_{m_j-m_{k-i}})$. Similarly with the factors $\GL_{n_{k-i}-n_{k-i+1}}(F)$. For a nontrivial complex character $\psi$ of $F$, denote
\begin{equation}\label{1.12.1}
	\psi_{\underline{p}}=\psi\circ b_{\underline{p}}.
\end{equation}
This is a complex character of $N^2_{\underline{p}}(F)$. 

Assume that $N^2_{\underline{p}}(F)$ is strictly contained in $N_{\underline{p}}(F)$. This is equivalent to the fact that there exist nonzero $m_i,n_j$ in the partition $\underline{p}$. Then $N_{\underline{p}}(F)/Ker(b_{\underline{p}})$ is isomorphic to the Heisenberg group, over $F$, $\mathcal{H}_{\underline{p}}(F)$, in $2(\sum\limits_{i=1}^km_in_i+\sum\limits_{i=2}^km_{i-1}n_i)+1$ variables. Indeed, we have the following symplectic form on $N_{\underline{p}}(F)/N^2_{\underline{p}}(F)$, 
\begin{equation}\label{1.13}
(xN^2_{\underline{p}}(F),yN^2_{\underline{p}}(F))=tr(\alpha_{\underline{p}}[x-I_\ell,y-I_\ell]).
\end{equation}
Note that $R_{\underline{p}}(F)$ preserves the symplectic form \eqref{1.13}, and hence it embeds in the corresponding symplectic group. Here is a polarization of the above symplectic space:
\begin{equation}\label{1.14}
N_{\underline{p}}(F)/N^2_{\underline{p}}(F)=\mathcal{X}_{\underline{p}}\oplus \mathcal{Y}_{\underline{p}},
\end{equation}
where in terms of the isomorphism $i_{\underline{p}}$,
\begin{equation}\label{1.15}
\mathcal{X}_{\underline{p}}(F)=\left\{xN^2_{\underline{p}}(F): x_{2j,2j+1}=0,\ 1\leq j\leq 2k-1\right\},
\end{equation}
$$
\mathcal{Y}_{\underline{p}}(F)=\left\{xN^2_{\underline{p}}(F): x_{2j-1,2j}=0,\  1\leq j\leq 2k-1\right\}.
$$
The center of $N_{\underline{p}}(F)/[N^2_{\underline{p}}(F),N^2_{\underline{p}}(F)]$ is $N^2_{\underline{p}}(F)/[N^2_{\underline{p}}(F),N^2_{\underline{p}}(F)]$, and the center of $N_{\underline{p}}(F)/Ker(b_{\underline{p}})$ is $\RZ_{\underline{p}}(F)=N^2_{\underline{p}}(F)/Ker(b_{\underline{p}})\cong F$. We may identify $\mathcal{X}_{\underline{p}}(F)$, $\mathcal{Y}_{\underline{p}}(F)$ with the following subgroups of $N_{\underline{p}}(F)$,  
\begin{equation}\label{1.16}
\RX_{\underline{p}}(F)=\left\{x\in N_{\underline{p}}(F): x_{i',j'}=0, \forall i'<j' \text{not of the form}\ (2j-1,2j), 1\leq j\leq 2k-1\right\},
\end{equation}
$$
\RY_{\underline{p}}(F)=\left\{x\in N_{\underline{p}}(F): x_{i',j'}=0, \forall i'<j' \text{not of the form}\ (2j,2j+1), 1\leq j\leq 2k-1\right\}.
$$
These are indeed (abelian) subgroups. Recall that we treat the elements of $N_{\underline{p}}(F)$ as block matrices according to the partition $\underline{p}'$ in \eqref{1.4}. The quotient map from $N_{\underline{p}}(F)$ onto $N_{\underline{p}}(F)/Ker(b_{\underline{p}})$ is injective on $\RX_{\underline{p}}(F)$ and $\RY_{\underline{p}}(F)$. Note that these subgroups stabilize $b_{\underline{p}}$.

The isomorphism $\iota'_{\underline{p}}:N_{\underline{p}}(F)/Ker(b_{\underline{p}})\cong \mathcal{H}_{\underline{p}}(F)$ is such that, for $u\in \RX_{\underline{p}}(F)$, $v\in \RY_{\underline{p}}(F)$, $z\in \RZ_{\underline{p}}(F)$,
\begin{equation}\label{1.17}
\iota'_{\underline{p}}(u)=(i_{\underline{p}}(uN^2_{\underline{p}}(F));0),
\end{equation}
$$
\iota'_{\underline{p}}(v)=(i_{\underline{p}}(vN^2_{\underline{p}}(F));0),
$$
$$
\iota'_{\underline{p}}(z)=(0;b_{\underline{p}}(z)).
$$
Denote by $\iota_{\underline{p}}:N_{\underline{p}}(F)\rightarrow \mathcal{H}_{\underline{p}}(F)$ the composition of $\iota'_{\underline{p}}$ and the quotient map from $N_{\underline{p}}(F)$ to $N_{\underline{p}}(F)/Ker(b_{\underline{p}})$.

\subsection{Fourier coefficients} Assume that $F$ is a number field and let $\BA$ denote its ring of adeles. We recall the Fourier coefficients attached to the partition $\underline{p}$ (and hence to its corresponding nilpotent orbit). We will consider three kinds of such Fourier coefficients.
 
 Let $\psi$ be a nontrivial character of $F\backslash \BA$. Consider the  homomorphism $b_{\underline{p}}: N^2_{\underline{p}}(\BA)\rightarrow \BA$, obtained by applying $b_{\underline{p}}$ above at every completion $F_v$ of $F$, for every place $v$ of $F$. As before, denote by $\psi_{\underline{p}}$ the character $\psi\circ b_{\underline{p}}$ of $N^2_{\underline{p}}(\BA)$.

Let $\varphi$ be an automorphic function on $\GL_\ell(\BA)$. Define, for $g\in \GL_\ell(\BA)$,
\begin{equation}\label{1.18}
\mathcal{F}_{\psi_{\underline{p}}}(\varphi)(g)=\int_{[N^2_{\underline{p}}]}\varphi(vg)\psi^{-1}_{\underline{p}}(v)dv.
\end{equation}
Note that for $\gamma\in R_{\underline{p}}(F)$, $g\in \GL_\ell(\BA)$,
\begin{equation}\label{1.19}
\mathcal{F}_{\psi_{\underline{p}}}(\varphi)(\gamma g)=\mathcal{F}_{\psi_{\underline{p}}}(\varphi)(g).
\end{equation}

Assume that $N_{\underline{p}}(F)$ strictly contains $N^2_{\underline{p}}(F)$. Then we saw that $N_{\underline{p}}(F)/Ker(b_{\underline{p}})$ is isomorphic to the Heisenberg group $\mathcal{H}_{\underline{p}}(F)$, corresponding to the symplectic space of dimension $2\mathsf{n}_{\underline{p}}$ over $F$, where
$$
\mathsf{n}_{\underline{p}}=\sum\limits_{i=1}^km_in_i+\sum\limits_{i=2}^km_{i-1}n_i.
$$
Define
\begin{equation}\label{1.20}
\mathcal{F}^1_{\psi_{\underline{p}}}(\varphi)(g)=\int_{[\RX_{\underline{p}}]}\int_{[N^2_{\underline{p}}]}\varphi(vxg)\psi^{-1}_{\underline{p}}(v)dvdx.
\end{equation}
Let $\theta_{\psi,2\mathsf{n}_{\underline{p}}}^\phi$ be a theta series corresponding to $\psi$, attached to the Weil representation of the semidirect product of the adelic Heisenberg group above and the double cover of the corresponding symplectic group, $\widetilde{\Sp}_{2\mathsf{n}_{\underline{p}}}(\BA)\mathcal{H}_{\underline{p}}(\BA)$. Here, $\phi$ lies in the Schwarz space $\mathcal{S}(\RX_{\underline{p}})(\BA)$ on $\RX_{\underline{p}}(\BA)$. Recall that $R_{\underline{p}}$ is embedded inside $\Sp_{2\mathsf{n}_{\underline{p}}}$. Denote by $\tilde{R}_{\underline{p}}(\BA)$ the inverse image of $R_{\underline{p}}(\BA)$ inside $\widetilde{\Sp}_{2\mathsf{n}_{\underline{p}}}(\BA)$. Define, for $\tilde{h}\in \tilde{R}_{\underline{p}}(\BA)$, projecting to $h\in R_{\underline{p}}(\BA)$,
\begin{equation}\label{1.21}
\mathcal{F}^2_{\psi_{\underline{p}},\phi}(\varphi)(\tilde{h}) =\int_{[N_{\underline{p}}]}\theta_{\psi,2\mathsf{n}_{\underline{p}}}^\phi(\iota_{\underline{p}}(v)\tilde{h})\varphi(vh)(v)dv.
\end{equation}

\begin{lem}\label{lem 1.1}
	
Let $\pi$ be an automorphic representation of $\GL_\ell(\BA)$ and consider the Fourier coefficients \eqref{1.18}, \eqref{1.20} and the Fourier-Jacobi coefficients \eqref{1.21}, as $\varphi$ varies in (the space of) $\pi$ (and in \eqref{1.21}, also as $\phi$ varies in $\mathcal{S}(\RX_{\underline{p}})(\BA)$).
Then $\mathcal{F}_{\psi_{\underline{p}}}$ is identically zero on $\pi$ iff $\mathcal{F}^1_{\psi_{\underline{p}}}$ is identically zero on $\pi$ iff $\mathcal{F}^2_{\psi_{\underline{p}},\phi}$ is zero on $\pi$, for all $\phi\in \mathcal{S}(\RX_{\underline{p}})(\BA)$.

\end{lem}

\begin{proof}
	
This is Lemma 1.1 in \cite{GRS03} in the case of general linear groups.

\end{proof}

\begin{rmk}\label{rmk 1.2}
	
In \eqref{1.20}, we	may consider
$$
\mathcal{F}'_{\psi_{\underline{p}}}(\varphi)(g)=\int_{[\RY_{\underline{p}}]}\int_{[N^2_{\underline{p}}]}\varphi(vyg)\psi^{-1}_{\underline{p}}(v)dvdy.
$$
We just exchanged the roles of $\RX_{\underline{p}}$ and $\RY_{\underline{p}}$.
Lemma \ref{lem 1.1} is valid with $\mathcal{F}'_{\psi_{\underline{p}}}(\varphi)$ instead of $\mathcal{F}^1_{\psi_{\underline{p}}}(\varphi)$.

\end{rmk}

For an automorphic representation $\pi$ of $\GL_\ell(\BA)$, let $\mathcal{O}(\pi)$ denote the set of maximal partitions $\underline{p}$, such that one of the coefficients \eqref{1.18}, \eqref{1.20}, \eqref{1.21}, and hence all three, is nontrivial on $\pi$. 

In this paper, we determine the positive poles
of an Eisenstein series, corresponding to $\Delta(\tau,m_1)|\cdot|^s\times  \Delta(\tau,m_2)|\cdot|^{-s}$. At each such pole, we consider the corresponding residual representation $\pi$ and compute $\mathcal{O}(\pi)$.

\subsection{Degenerate Whittaker coefficients and generalized Whittaker coefficients} Following the terminology of \cite{GGS17}, with a slight difference, $(h_{\underline{p}},\alpha_{\underline{p}})$ is a Whittaker pair.
We recall that for a one parameter subgroup, $h:F^*\rightarrow \GL_\ell(F)$, and a nilpotent element $\alpha\in M_\ell(F)$, such that
\begin{equation}\label{1.22}
h(t)\alpha h(t)^{-1}=t^{-2}\alpha,
\end{equation}
the pair $(h,\alpha)$ is called a Whittaker pair. Moreover, one can check that $h_{\underline{p}}$ is, in the terminology of loc. cit., a neutral element for $\alpha_{\underline{p}}$. This means the following. Let $\mathfrak{h}_{\underline{p}}$ be the diagonal matrix in $M_\ell(F)$ obtained by taking the exponents of the powers of $t$ along the diagonal of $h_{\underline{p}}(t)$. Then there is an element $\beta\in M_\ell(F)$, such that $(\alpha_{\underline{p}},\mathfrak{h}_{\underline{p}},\beta)$ is an $sl_2$-triple. For a given Whittaker pair $(h,\alpha)$, one associates, as in \cite{MW87}, the following unipotent subgroup $N'_{\alpha,h}(F)$ of $\GL_\ell(F)$. For $i\geq 1$, let
\begin{equation}\label{1.23}
\mathfrak{g}^h_i=\left\{x\in M_\ell(F): h(t)xh(t)^{-1}=t^ix\right\}.
\end{equation}
Let $\mathfrak{n}^h=\sum_{i\geq 1}\mathfrak{g}^h_i$, $\mathfrak{g}^h_{\geq2}=\sum_{i\geq 2}\mathfrak{g}^h_i$, $N^2(F)=exp(\mathfrak{g}^h_{\geq2})$. Denote by $\mathfrak{g}_\alpha$ the centralizer of $\alpha$ in the Lie algebra $M_\ell(F)$, and denote $N'_{h,\alpha}(F)=N^2(F)exp(\mathfrak{g}^h_1\cap \mathfrak{g}_\alpha)$. See \cite{MW87}, Sec. I.7. Define the homomorphism $b_{\alpha,h}:N'_{\alpha,h}(F)\rightarrow F$,
$$
b_{\alpha,h}(x)=tr(\alpha\cdot log (x)).
$$
In our case, we have $N'_{\alpha_{\underline{p}},h_{\underline{p}}}=N^2_{\underline{p}}$. When $F$ is a number field, keep denoting by $b_{\alpha,h}$ the homomorphism of $N'_{\alpha,h}(\BA)$ deduced from $b_{\alpha,h}$ at every place of $F$. Let $\psi$ be a nontrivial character of $F\backslash \BA$. Denote $\psi_{\alpha,h}=\psi\circ b_{\alpha,h}$. Consider the following Fourier coefficient, applied to an automorphic function $\varphi$ on $\GL_\ell(\BA)$,
$$
\mathcal{F}_{\alpha,h}(\varphi)(g)=\int_{[N'_{\alpha,h}]}\varphi(vg)\psi^{-1}_{\alpha,h}(v)dv.
$$
Then we have the following theorem of Gourevitch, Gomez and Sahi, Theorm C in \cite{GGS17},

\begin{thm}\label{thm 1.3}
	
Let $\pi$ be an automorphic representation of $\GL_\ell(\BA)$. Let $(\alpha,h)$, $(\alpha,h')$ be Whittaker pairs, such that $h$ is neutral for $\alpha$. Assume that $\mathcal{F}_{\alpha,h'}$ is nontrivial on $\pi$. Then $\mathcal{F}_{\alpha,h}$ is nontrivial on $\pi$.

\end{thm}
This theorem is used effectively by Liu and Xu in \cite{LX20}. In \cite{LX20}, the authors consider the set $\mathfrak{n}^m(\pi)$ of maximal nilpotent orbits $\mathcal{O}\subset M_\ell(F)$, under the adjoint action of $\GL_\ell(F)$, such that there is a Whittaker pair $(\alpha,h)$, $\alpha\in \mathcal{O}$, $h$  neutral for $\alpha$ and 
$\mathcal{F}_{\alpha,h}$ is nontrivial on $\pi$. This property is independent of the choice of $h$ which is neutral for $\alpha$. Liu and Xu denote by $\mathfrak{p}^m(\pi)$ the partitions of $\ell$ corresponding to the elements in $\mathfrak{n}^m(\pi)$. It follows that the sets $\mathcal{O}(\pi)$ and $\mathfrak{p}(\pi)$ coincide.

\subsection{Fourier coefficients with respect to $\underline{p}=(k,1^{\ell-k})$}
We consider the example of the partition $\underline{p}=(k,1^{\ell-k})$ of $\ell$. We will relate the corresponding Fourier coefficients to the Fourier coefficients which are global analogs of Jacquet modules considered by Bernstein-Zelevinsky \cite{BZ77},
\begin{equation}\label {1.24}
\mathcal{D}^0_{\psi,k}(\varphi)(g)=\int_{[V^0_{\ell-k,1^k}]}\varphi(vg)\psi^{-1}_{V_{\ell-k,1^k}}(v)dv,
\end{equation}
where $V^0_{\ell-k,1^k}=V_{\ell-k+1,1^{k-1}}$, and for
$$
v=\begin{pmatrix}I_{\ell-k}&x\\&z\end{pmatrix}\in V^0_{\ell-k,1^k}(\BA),\ \ \
\psi_{V_{\ell-k,1^k}}(v)=\psi_{Z_k}(z)=\psi(z_{1,2}+z_{2,3}+\cdots+z_{k-1,k}).
$$
Note that the first column of $x$ is zero. The global analog of the Bernstein-Zelevinsky derivative is
\begin{equation}\label{1.25}
\mathcal{D}_{\psi,k}(\varphi)(g)=\int_{[V_{\ell-k,1^k}]}\varphi(vg)\psi^{-1}_{V_{\ell-k,1^k}}(v)dv,
\end{equation}
where $\psi_{V_{\ell-k,1^k}}(v)$ is given by the same formula, i.e.
$$
\psi_{V_{\ell-k,1^k}}(\begin{pmatrix}I_{\ell-k}&x\\&z\end{pmatrix})=\psi_{Z_k}(z).
$$

\begin{prop}\label{prop 1.4}
	
Let $\pi$ be an automorphic representation of $\GL_\ell(\BA)$. Let $1\leq k\leq \ell$. Then $\mathcal{F}_{\psi_{(k,1^{\ell-k})}}$ is nontrivial on $\pi$ iff $\mathcal{D}^0_{\psi,k}$ is nontrivial on $\pi$. Moreover, there is an explicit formula relating the two Fourier coefficients.

\end{prop}

\begin{proof}
	
The Fourier coefficient \eqref{1.24} corresponds to the following Whittaker pair $(\alpha,h)$:
\begin{equation}\label{1.26}
h(t)=\diag(t^{k-1}I_{\ell-k+1},t^{k-3},t^{k-5},...,t^{1-k}),
\end{equation}
$$
\alpha=\begin{pmatrix}0_{(\ell-k)\times(\ell-k)}\\&0\\&1&0\\&&1&0\\&&&&\ddots\\&&&&&1&0\end{pmatrix}.
$$
Let $\underline{p}=(k,1^{\ell-k})$. Then $\alpha_{\underline{p}}$ is conjugate to $\alpha$. They have the same Jordan normal form. By Theorem \ref{thm 1.3}, if
$\mathcal{D}^0_{\psi,k}$ is nontrivial on $\pi$, then $\mathcal{F}_{\psi_{\underline{p}}}$ is nontrivial on $\pi$. We will prove the converse by exchanging roots (\cite{GRS11}, Lemma 7.1), showing an explicit passage from $\mathcal{F}_{\psi_{\underline{p}}}$ to $\mathcal{D}^0_{\psi,k}$. This will automatically prove both directions of the proposition. By definition, we have to consider two cases according to the parity of $k$. Let us start with the case where $k$ is odd. Then in this case,
$$
N_{\underline{p}}=N^2_{\underline{p}}=V_{1^{\frac{k-1}{2}},\ell-k+1,1^{\frac{k-1}{2}}},
$$
and for $v\in N_{\underline{p}}(\BA)$,
\begin{equation}\label{1.27}
\psi_{\underline{p}}(v)=\psi(\sum\limits_{i=1}^{\frac{k-1}{2}}v_{i,i+1}+v_{\frac{k+1}{2},\ell-\frac{k-3}{2}}+\sum\limits_{j=\ell-\frac{k-3}{2}}^{\ell-1}v_{j,j+1}).
\end{equation}
Let $w_0=\begin{pmatrix}I_{\frac{k-1}{2}}\\&0&I_{\ell-k}\\&1&0\\&&&I_{\frac{k-1}{2}}\end{pmatrix}$. Then conjugating by $w_0$ inside the integral defining $\mathcal{F}_{\psi_{\underline{p}}}(\varphi)(g)$, we get
\begin{equation}\label{1.28}
\mathcal{F}_{\psi_{\underline{p}}}(\varphi)(g)=\int_{[V_{1^{\frac{k-1}{2}},\ell-k+1,1^{\frac{k-1}{2}}}]}\varphi(vw_0g)\psi^{-1}_0(v)dv,
\end{equation}
where, for $v\in V_{1^{\frac{k-1}{2}},\ell-k+1,1^{\frac{k-1}{2}}}(\BA)$,
\begin{equation}\label{1.29}
\psi_0(v)=\psi(\sum\limits_{i=1}^{\frac{k-3}{2}}v_{i,i+1}+v_{\frac{k-1}{2},\ell-\frac{k-1}{2}}+\sum\limits_{j=\ell-\frac{k-1}{2}}^{\ell-1}v_{j,j+1}).
\end{equation}
Let $w_1=\begin{pmatrix}I_{\frac{k-3}{2}}\\&0&I_{\ell-k}\\&1&0\\&&&I_{\frac{k+1}{2}}\end{pmatrix}$. Then conjugating by $w_1$ inside \eqref{1.28}, we get
\begin{equation}\label{1.30}
\mathcal{F}_{\psi_{\underline{p}}}(\varphi)(g)=\int_{[V^{w_1}_{1^{\frac{k-1}{2}},\ell-k+1,1^{\frac{k-1}{2}}}]}\varphi(vw_1w_0g)\psi^{-1}_{0,1}(v)dv.
\end{equation}
Here, an element $v\in V^{w_1}_{1^{\frac{k-1}{2}},\ell-k+1,1^{\frac{k-1}{2}}}(\BA)=w_1V_{1^{\frac{k-1}{2}},\ell-k+1,1^{\frac{k-1}{2}}}(\BA)w_1^{-1}$ has the form
$$
v=\begin{pmatrix}z_1&\ast&e&\ast&\ast\\&I_{\ell-k}&0&0&\ast\\&y&1&u&\ast\\&&&1&u'\\&&&&z_2\end{pmatrix},\ z_1\in Z_{\frac{k-3}{2}}(\BA),\ z_2\in Z_{\frac{k-1}{2}}(\BA),
$$
and
$$
\psi_{0,1}(v)=\psi_{Z_{\frac{k-3}{2}}}(z_1)\psi_{Z_{\frac{k+1}{2}}}(\begin{pmatrix}1&u'\\&z_2\end{pmatrix})\psi(e_{\frac{k-3}{2}}+u).
$$
Let 
$$
X_1=\left\{\begin{pmatrix}I_{\frac{k-3}{2}}\\&I_{\ell-k}&0&x\\&&1&0\\&&&1\\&&&&I_{\frac{k-1}{2}}\end{pmatrix}\right\},\ Y_1=\left\{\begin{pmatrix}I_{\frac{k-3}{2}}\\&I_{\ell-k}\\&y&1\\&&&I_{\frac{k+1}{2}}\end{pmatrix}\right\}.
$$
Now, we can exchange roots between $Y_1$ and $X_1$ in \eqref{1.30}. See \cite{GRS11}, Lemma 7.1. We get
\begin{equation}\label{1.31}
\mathcal{F}_{\psi_{\underline{p}}}(\varphi)(g)=\int_{Y_1(\BA)}\int_{[V_{1^{\frac{k-3}{2}},\ell-k+1,1^{\frac{k+1}{2}}}]}\varphi(vy_1w_1w_0g)\psi^{-1}_1(v)dvdy_1,
\end{equation}
where, for $v\in V_{1^{\frac{k-3}{2}},\ell-k+1,1^{\frac{k+1}{2}}}(\BA)$,
\begin{equation}\label{1.32}
\psi_1(v)=\psi(\sum\limits_{i=1}^{\frac{k-5}{2}}v_{i,i+1}+v_{\frac{k-3}{2},\ell-\frac{k+1}{2}}+\sum\limits_{j=\ell-\frac{k+1}{2}}^{\ell-1}v_{j,j+1}).
\end{equation}
Note that in the passage from the integral \eqref{1.28} to the $dv$-integration in \eqref{1.31}, we passed from the partition $(1^{\frac{k-1}{2}},\ell-k+1,1^{\frac{k-1}{2}})$ to the partition $(1^{\frac{k-3}{2}},\ell-k+1,1^{\frac{k+1}{2}})$; $1$ is taken out once from the left and is added to the right. We note a similar passage from $\psi_0$ in \eqref{1.29} to $\psi_1$ in \eqref{1.32}. Thus, we may repeat the argument, passing at each step from the partition $(1^{\frac{k-2i+1}{2}},\ell-k+1,1^{\frac{k+2i-3}{2}})$ to $(1^{\frac{k-2i-1}{2}},\ell-k+1,1^{\frac{k+2i-1}{2}})$, using conjugation in the inner $dv$-integration by $w_i=\diag(I_{\frac{k-2i-1}{2}}, \begin{pmatrix}&I_{\ell-k}\\1\end{pmatrix},I_{\frac{k+2i-1}{2}})$. The root exchange is between the following subgroups
$$
X_i=\left\{\begin{pmatrix}I_{\frac{k-2i-1}{2}}\\&I_{\ell-k}&0&x\\&&1&0\\&&&1\\&&&&I_{\frac{k+2i-3}{2}}\end{pmatrix}\right\},\ Y_i=\left\{\begin{pmatrix}I_{\frac{k-2i-1}{2}}\\&I_{\ell-k}\\&y&1\\&&&I_{\frac{k+2i-1}{2}}\end{pmatrix}\right\}.
$$
This contributes an outer integration along $Y_i(\BA)$ and an inner  $dv$-integration, which is the Fourier coefficient along $[V_{1^{\frac{k-2i-1}{2}},\ell-k+1,1^{\frac{k+2i-1}{2}}}]$ with repect to the character
$$
\psi_i(v)=\psi(\sum\limits_{j=1}^{\frac{k-2i-3}{2}}v_{j,j+1}+v_{\frac{k-2i-1}{2},\ell-\frac{k+2i-1}{2}}+\sum\limits_{j=\ell-\frac{k+2i-1}{2}}^{\ell-1}v_{j,j+1}).
$$Note that
\begin{equation}\label{1.33}
w_iY_{i-1}w_i^{-1}=\left\{\diag(I_{\frac{k-2i-1}{2}},\begin{pmatrix}I_{\ell-k}\\0&1\\y&0&1\end{pmatrix},I_{\frac{k+2i-3}{2}})\right\}.
\end{equation}
We continue until we reach the partition $(\ell-k+1,k-1)$, where we get, using \eqref{1.33},
\begin{equation}\label{1.34}
\mathcal{F}_{\psi_{\underline{p}}}(\varphi)(g)=\int_{Y(\BA)}\mathcal{D}^0_{\psi,k}(\varphi)(ywg)dy,
\end{equation}
where 
$$
w=w_{\frac{k-1}{2}}w_{\frac{k-3}{2}}\cdots w_2w_1w_0=\begin{pmatrix}&I_{\ell-k}\\I_{\frac{k+1}{2}}\\&&I_{\frac{k-1}{2}}\end{pmatrix},
$$
$$
Y=\left\{\begin{pmatrix}I_{\ell-k}\\y&I_{\frac{k-1}{2}}\\&&I_{\frac{k+1}{2}}\end{pmatrix}\right\}.
$$
Applying \cite{GRS11}, Cor. 7.1, at each step above, we conclude that $\mathcal{F}_{\psi_{\underline{p}}}$ is nontrivial on $\pi$ iff $\mathcal{D}^0_{\psi,k}$ is nontrivial on $\pi$. This proves the proposition in case $k$ is odd. 

Assume that $k$ is even. Keep denoting $\underline{p}=(k,1^{\ell-k})$. In this case, $N^2_{\underline{p}}(F)$ is strictly contained in $N_{\underline{p}}(F)$. By Lemma \ref{lem 1.1}, $\mathcal{F}_{\underline{p}}$ is nontrivial on $\pi$ iff $\mathcal{F}^1_{\underline{p}}$ (see \eqref{1.20}) is nontrivial on $\pi$. Denote $N'_{\underline{p}}=N^2_{\underline{p}}\RX_{\underline{p}}$. Then
$N'_{\underline{p}}=V_{1^{\frac{k}{2}},\ell-k+1,1^{\frac{k-2}{2}}}$. Denote by $\psi_0$ the character of $N'_{\underline{p}}(\BA)$, obtained as the extension of $\psi_{\underline{p}}$ on $N^2_{\underline{p}}(\BA)$
by the trivial character on $\RX_{\underline{p}}(\BA)$. Thus, for $v\in V_{1^{\frac{k}{2}},\ell-k+1,1^{\frac{k-2}{2}}}(\BA)$,
$$
\psi_0(v)=\psi(\sum\limits_{i=1}^{\frac{k-2}{2}}v_{i,i+1}+v_{\frac{k}{2},\ell-\frac{k-2}{2}}+\sum\limits_{j=\ell-\frac{k-2}{2}}^{\ell-1}v_{j,j+1}),
$$
and
\begin{equation}\label{1.35}
\mathcal{F}^1_{\underline{p}}(\varphi)(g)=\int_{[V_{1^{\frac{k}{2}},\ell-k+1,1^{\frac{k-2}{2}}}]}\varphi(vg)\psi_0^{-1}(v)dv.
\end{equation}
Now, we are at a similar situation as in \eqref{1.28} (replacing $w_0$ there by $I_\ell$). We carry very similar steps and get
\begin{equation}\label{1.36}
\mathcal{F}^1_{\psi_{\underline{p}}}(\varphi)(g)=\int_{Y(\BA)}\mathcal{D}^0_{\psi,k}(\varphi)(ywg)dy,
\end{equation}
where 
$$
w=\begin{pmatrix}&I_{\ell-k}\\I_{\frac{k}{2}}\\&&I_{\frac{k}{2}}\end{pmatrix},
$$
$$
Y=\left\{\begin{pmatrix}I_{\ell-k}\\y&I_{\frac{k}{2}}\\&&I_{\frac{k}{2}}\end{pmatrix}\right\}.
$$
As before, we conclude that $\mathcal{F}^1_{\psi_{\underline{p}}}$, and hence  $\mathcal{F}_{\psi_{\underline{p}}}$, is nontrivial on $\pi$ iff $\mathcal{D}^0_{\psi,k}$ is nontrivial on $\pi$. This proves the proposition in case $k$ is even. 

\end{proof}

Let us check how the action of the stabilizer $R_{\underline{p}}(\BA)$ of $\psi_{\underline{p}}$, $\underline{p}=(k,1^{\ell-k})$, is translated via the identities \eqref{1.34}, \eqref{1.36}. Assume that $k$ is odd. Then the stabilizer $R_{\underline{p}}(\BA)$ inside $M_{1^{\frac{k-1}{2}},\ell-k+1,1^{\frac{k-1}{2}}}(\BA)$ is the adele points of
\begin{equation}\label{1.37}
R_{\underline{p}}=\left\{\diag(tI_{\frac{k+1}{2}},a,tI_{\frac{k-1}{2}}): a\in \GL_{\ell-k}, t\in \GL_1\right\}\cong \GL_{\ell-k}\times \GL_1.
\end{equation}
Let $\tilde{R}_{\underline{p}}$ denote the subgroup of $R_{\underline{p}}$ consisting of the elements in \eqref{1.37} with $t=1$. Write its elements as $\tilde{r}_{\underline{p}}(a)=\diag(I_{\frac{k+1}{2}},a,I_{\frac{k-1}{2}})$. Denote, also, $r_{\underline{p}}(a)=\diag(a,I_k)$. Then in \eqref{1.34}, we have, for $a\in \GL_{\ell-k}(\BA)$,
\begin{equation}\label{1.38}
\mathcal{F}_{\psi_{\underline{p}}}(\varphi)(\tilde{r}_{\underline{p}}(a)g)=|\det(a)|^{\frac{1-k}{2}}\int_{Y(\BA)}\mathcal{D}^0_{\psi,k}(\varphi)(r_{\underline{p}}(a)ywg)dy.
\end{equation}
Assume that $k$ is even. We consider the analogous objects for $\mathcal{F}^1_{\psi_{\underline{p}}}(\varphi)$. Here,
\begin{equation}\label{1.39}
\tilde{R}_{\underline{p}}=\left\{\tilde{r}_{\underline{p}}(a)=\diag(I_{\frac{k}{2}},a,I_{\frac{k}{2}}): a\in \GL_{\ell-k}\right\}\cong \GL_{\ell-k}.
\end{equation}
Denote $r_{\underline{p}}(a)=\diag(a,I_k)$ (as in the preious case). Then in \eqref{1.36}, we have, for $a\in \GL_{\ell-k}(\BA)$,
\begin{equation}\label{1.40}
\mathcal{F}^1_{\psi_{\underline{p}}}(\varphi)(\tilde{r}_{\underline{p}}(a)g)=|\det(a)|^{-\frac{k}{2}}\int_{Y(\BA)}\mathcal{D}^0_{\psi,k}(\varphi)(r_{\underline{p}}(a)ywg)dy.
\end{equation}

We have a similar relation of $\mathcal{F}_{\psi_{\underline{p}}}$ and the following analog of the Bernstein-Zelevinsky coefficient \eqref{1.24},
\begin{equation}\label{1.41}
^0\mathcal{D}'_{\psi,k}(\varphi)(g)=\int_{[V^0_{1^k,\ell-k}]}\varphi(vg)\psi^{-1}_{V_{1^k,\ell-k}}(v)dv,
	\end{equation}
where $V^0_{1^k,\ell-k}=V_{1^{k-1},\ell-k+1}$, and for $v=\begin{pmatrix}z&x\\&I_{\ell-k}\end{pmatrix}\in V^0_{1^k,\ell-k}(\BA)$, $\psi_{V_{1^k,\ell-k}}(v)=\psi_{Z_k}(z)$. Note that the last row of $x$ is zero.
We have the following analog of Prop. \ref{prop 1.4}, \eqref{1.34} and \eqref{1.35}, with a similar proof, which we leave to the reader.

\begin{prop}\label{prop 1.5}
	
Let $\pi$ be an automorphic representation of $\GL_\ell(\BA)$. Let $1\leq k\leq \ell$. Then $\mathcal{F}_{\psi_{(k,1^{\ell-k})}}$ is nontrivial on $\pi$ iff $^0\mathcal{D}'_{\psi,k}$ is nontrivial on $\pi$. Moreover, we have the following identities (analogs of \eqref{1.34}, \eqref{1.36}).\\
Assume that $k$ is odd. Then
\begin{equation}\label{1.42}
\mathcal{F}_{\psi_{\underline{p}}}(\varphi)(g)=\int_{Y'(\BA)}{}^0\mathcal{D}'_{\psi,k}(\varphi)(yw'g)dy,
\end{equation}
where 
$$
w'=\begin{pmatrix}I_{\frac{k+1}{2}}\\&&I_{\frac{k-1}{2}}\\&I_{\ell-k}\end{pmatrix},
$$
$$
Y'=\left\{\begin{pmatrix}I_{\frac{k+1}{2}}\\&I_{\frac{k-1}{2}}\\&y&I_{\ell-k}\end{pmatrix}\right\}.
$$
Assume that $k$ is even. Then
\begin{equation}\label{1.43}
\mathcal{F}^1_{\psi_{\underline{p}}}(\varphi)(g)=\int_{Y'(\BA)}{}^0\mathcal{D}'_{\psi,k}(\varphi)(yw'g)dy,
\end{equation}
where 
$$
w'=\begin{pmatrix}I_{\frac{k}{2}}\\&&I_{\frac{k}{2}}\\&I_{\ell-k}\end{pmatrix},
$$
$$
Y'=\left\{\begin{pmatrix}I_{\frac{k+2}{2}}\\&I_{\frac{k-2}{2}}\\&y&I_{\ell-k}\end{pmatrix}\right\}.
$$	 
	
\end{prop}

Finally, we have the following analogs of \eqref{1.38}, \eqref{1.40}. Denote, for $a\in \GL_{\ell-k}(\BA)$, $r'_{\underline{p}}(a)=\diag(I_k,a)$. Then in \eqref{1.41}, we have, 
\begin{equation}\label{1.44}
\mathcal{F}_{\psi_{\underline{p}}}(\varphi)(\tilde{r}_{\underline{p}}(a)g)=|\det(a)|^{\frac{k-1}{2}}\int_{Y'(\BA)}{}^0\mathcal{D}'_{\psi,k}(\varphi)(r'_{\underline{p}}(a)yw'g)dy;
\end{equation}
and in \eqref{1.43}, we have, 
\begin{equation}\label{1.45}
\mathcal{F}^1_{\psi_{\underline{p}}}(\varphi)(\tilde{r}_{\underline{p}}(a)g)=|\det(a)|^{\frac{k}{2}}\int_{Y'(\BA)}{}^0\mathcal{D}'_{\psi,k}(\varphi)(r'_{\underline{p}}(a)yw'g)dy.
\end{equation}

\subsection{Fourier coefficients with respect to $(k^r)$}
Another example considered in this paper is the partition $\underline{p}=(k^r)$ of $\ell=kr$.
Here, $N_{\underline{p}}=N_{\underline{p}}^2=V_{r^k}$. Write $v\in V_{r^k}(\BA)$ as
$$
v=\begin{pmatrix}I_r&x_{1,2}&\cdots&x_{1,k}\\&I_r&\cdots&x_{2,k}\\&&\ddots\\&&&I_r\end{pmatrix}.
$$
Then
\begin{equation}\label{1.46}
\psi_{\underline{p}}(v)=\psi(tr(x_{1,2}+x_{2,3}+\cdots+x_{k-1,k})).
\end{equation}
Consider the following character of $Z_\ell(\BA)=V_{1^\ell}(\BA)$. 
\begin{equation}\label{1.47}
\psi_{k,r}(z)=\sum\limits_{i=1}^{\ell-1}\psi(\delta_iz_{i,i+1}),
\end{equation}
where $\delta_{jk}=0$, for $1\leq j\leq r-1$, and otherwise $\delta_i=1$. Consider the Fourier coefficient, 
\begin{equation}\label{1.48}
\mathcal{Z}_{\psi,k,r}(\varphi)(g)=\int\limits_{[Z_\ell]}\varphi(zg)\psi^{-1}_{k,r}(z)dz.
\end{equation}
Note that $\mathcal{Z}_{\psi,k,r}(\varphi)$ is obtained first by taking the constant term of $\varphi$ along $V_{k^r}$ and then applying the Whittaker coefficient on the maximal upper unipotent subgroups $Z_k(\BA)$ on each one of the $r$ blocks inside the Levi subgroup $M_{k^r}(\BA)$.
\begin{thm}\label{thm 1.6}
Let $\pi$ be an automorphic representation of $\GL_\ell(\BA)$, $\ell=kr$. Assume that $\mathcal{Z}_{\psi,k,r}$ is nontrivial on $\pi$. Then $\mathcal{F}_{\psi_{(k^r)}}$ is nontrivial on $\pi$. If $\mathcal{O}(\pi)=(k^r)$, then $\mathcal{Z}_{\psi,k,r}$ is nontrivial on $\pi$. Moreover, there is an explicit formula relating the two Fourier coefficients.
\end{thm}
\begin{proof}
The first part is similar to the first part of Prop. \ref{prop 1.4}. The Fourier coefficient \eqref{1.48} corresponds to the following Whittaker pair $(\alpha,h)$:
\begin{equation}\label{1.49}
h(t)=\diag(t^{\ell-1},t^{\ell-3},...,t^{1-\ell}),
\end{equation}
$$
\alpha=diag(J_k,...,J_k),\ (r\  \text{times}),\ J_k=\begin{pmatrix}0\\1&0\\&1&0\\&&&\ddots\\&&&1&0\end{pmatrix}_{k\times k}.
$$
Let $\underline{p}=(k^r)$. Then $\alpha_{\underline{p}}$ is conjugate to $\alpha$. They have the same Jordan normal form. Thus, the first part of the proposition follows from Theorem \ref{thm 1.3}.	
	
Assume that $\mathcal{O}(\pi)=(k^r)$. As in Prop. \ref{prop 1.4}, we will prove the second part of the proposition by exchanging roots, which will also give us an explicit formula relating the two Fourier coefficients. We will need this formula later.
	
By assumption, $\mathcal{F}_{\psi_{(k^r)}}$ is nontrivial on $\pi$. Let
$$
w_0=\begin{pmatrix}e_{1,1}&\cdots&e_{k,1}\\\vdots&&\vdots\\e_{1,r}&\cdots&e_{k,r}\end{pmatrix},
$$
where $\{e_{i,j}:\ 1\leq i\leq k,\ 1\leq j\leq r\}$, denotes the standard basis of $M_{k\times r}(F)$. As in \eqref{1.30}, for $\varphi$ in the space of $\pi$,  
\begin{equation}\label{1.50}
\mathcal{F}_{\psi_{\underline{p}}}(\varphi)(g)=\int_{[V_{r^k}^{w_0}]}\varphi(vw_0g)\psi_0^{-1}(v)dv.
\end{equation}
Here, an element of $V_{r^k}^{w_0}(\BA)=w_0V_{r^k}(\BA)w^{-1}_0$ has the form
\begin{equation}\label{1.51}
v=\begin{pmatrix}X_{1,1}&\cdots&X_{1,r}\\ \vdots&&\vdots\\X_{r,1}&\cdots&X_{r,r}\end{pmatrix},
\end{equation}
where, for $1\leq i\neq j\leq r$, $X_{i,i}\in Z_k(\BA)$ and $X_{i,j}$ is upper nilpotent in $M_{k\times k}(\BA)$. Denote
\begin{equation}\label{1.52}
X_{i,i}=\begin{pmatrix}1&x^{i,i}_{1,2}&\cdots&x^{i,i}_{1,k}\\&1&\cdots&x^{i,i}_{2,k}\\&&\ddots\\&&&1\end{pmatrix},\ \ X_{i,j}=\begin{pmatrix}0&x^{i,j}_{1,2}&\cdots&x^{i,j}_{1,k}\\&0&\cdots&x^{i,j}_{2,k}\\&&\ddots\\&&&0\end{pmatrix}.
\end{equation}
Then
\begin{equation}\label{1.52.1}
\psi_0(v)=\prod\limits_{i=1}^r\psi(\sum\limits_{j=1}^{k-1}(x^{i,i}_{j,j+1})).
\end{equation}
Let, for $1\leq i< j\leq r$, $Y^1_{i,j}$ denote the subgroup of elements \eqref{1.51}, such that all blocks are the blocks of the identity matrix except those below the diagonal in the first column, and, for $2\leq i'\leq r$, we have $x^{i',1}_{i'',j''}=0$, for $(i'',j'')\neq (i,j)$. Let $X_1^{j-1,i}$ be the subgroup of elements $z\in Z_\ell$, such that when we write $z$ as a matrix of blocks as in \eqref{1.51}, then all blocks are the blocks of the identity matrix, except above the diagonal in the first row, and, for $2\leq j'\leq r$, $x^{1,j'}_{i'',j''}=0$, for all $(i'',j'')\neq (j-1,i)$. Then we can exchange roots $Y^1_{1,2}$ with $X_1^{1,1}$, then $Y^1_{2,3}$ with $X_1^{2,2}$, next $Y^1_{1,3}$ with $X_1^{2,1}$, and so on. In step $j$, we exchange $Y^1_{i,j}$ and $X^1_{j-1,i}$, for $i=j-1,j-2,...,1$, in this order. By \cite{GRS11}, Lemma 7.1,
\begin{equation}\label{1.53}
\mathcal{F}_{\psi_{\underline{p}}}(\varphi)(g)=\int\limits_{Y^1(\BA)}\int_{[V^1_0]}\varphi(vyw_0g)\psi_{1,0}^{-1}(v)dvdy,
\end{equation}
where $Y^1$ is the subgroup generated by $Y^1_{i,j}$, $1\leq i< j\leq r$, and $V^1_0$ is the subgroup generated by the  elements of $V_{r^k}^{w_0}$ with zero first block column below the diagonal and $X_1^{j-1,i}$, $1\leq i\leq j\leq r$. The character $\psi_{1,0}$ is given by the r.h.s. of \eqref{1.52.1}. When we write the elements $v$ of $V_0^1$ as block matrices as in \eqref{1.51}, we have $X_{i,1}=0$, for $2\leq i\leq r$, and each block in the first row is such that its last row is zero (and the other rows are arbitrary). All the remaining blocks are exactly as in \eqref{1.51}, \eqref{1.52}. Denote
$$
x_1(e)=\begin{pmatrix}I_{(k-1)}\\&1&e\\&&I_{(r-1)k}\end{pmatrix}.
$$	
Let $V^1$ be the subgroup generated by $V^1_0$ and all elements $x_1(e)$. Denote by $\psi_1$ the trivial extension of $\psi_{1,0}$. We claim that 
\begin{equation}\label{1.54}
\int_{[V^1_0]}\varphi(vg')\psi_{1,0}^{-1}(v)dv=\int_{[V^1]}\varphi(vg')\psi_1^{-1}(v)dv.
\end{equation}
For this, we may asume that $g'=I_\ell$. Note that $V^0_{1^k,\ell-k}$ (see \eqref{1.41}) is a subgroup of $V^1_0$ and that the restricton of $\psi_{0,1}$ to $V^0_{1^k,\ell-k}(\BA)$ is $\psi_{V^0_{1^k,\ell-k}}$ from \eqref{1.41}. It is then enough to show that, for all $\varphi\in \pi$, 
\begin{equation}\label{1.55}
{}^0\mathcal{D}'_{\psi,k}(\varphi)(I_\ell)=\mathcal{D}'_{\psi,k}(\varphi)(I_\ell).
\end{equation}
Indeed, consider the Fourier expansion of the following smooth function on\\ $F^{(r-1)k}\backslash \BA^{(r-1)k}$, $e\mapsto {}^0\mathcal{D}'_{\psi,k}(\varphi)(x_1(e))$.
Then only the trivial character contributes to this expansion. Indeed, it is enough to show that 
$$
\int\limits_{F^{(r-1)k}\backslash \BA^{(r-1)k}}{}^0\mathcal{D}'_{\psi,k}(\varphi)(x_1(e)\psi^{-1}(e_1))de=0,
$$
for all $\varphi\in \pi$. The last integral is ${}^0\mathcal{D}'_{\psi,k+1}(\varphi)(I_\ell)$. If this is nontrivial on $\pi$, then by Prop. \ref{prop 1.5}, the Fourier coefficient corresponding to $(k+1,1^\ell-k-1)$ is nontrivial on $\pi$, contradicting our assumption. This proves \eqref{1.55}. 

We continue by induction. Let, for  $1< c\leq r-1$, $Y^c$ be the subgroup defined analogously to $Y^1$, only that we take the elements \eqref{1.51}, where all blocks are the blocks of the identity matrix except column $c$ below the diagonal. Let $V_0^c$ be the subgroup generated by $V^{c-1}$ and the subgroups $X_c^{j-1,i}$, $1\leq i<j\leq r$, defined analogously to $X_1^{j-1,i}$, where we use the blocks in row $c$ above the diagonal. As in \eqref{1.53}, we get
\begin{equation}\label{1.56}
\mathcal{F}_{\psi_{\underline{p}}}(\varphi)(g)=\int\limits_{Y^1(\BA)}\cdots\int\limits_{Y^c(\BA)} \int_{[V^c_0]}\varphi(vy_c\cdots y_1w_0g)\psi_{c,0}^{-1}(v)dvdy_c\cdots dy_1,
\end{equation}
where $\psi_{c,0}$ is given by the r.h.s. of \eqref{1.52.1}. Denote	
$$
x_c(e)=\begin{pmatrix}I_{ck-1}\\&1&e\\&&I_{(r-c)k}\end{pmatrix}.
$$
Let $V^c$ be the subgroup generated by $V_0^c$ and all elemnts $x_c(e)$. Denote by $\psi_c$ the trivial extension of $\psi_{c,0}$. As in \eqref{1.54}, for all $\varphi\in \pi$,	
\begin{equation}\label{1.57}
\int_{[V^c_0]}\varphi(v)\psi_{c,0}^{-1}(v)dv=\int_{[V^c]}\varphi(v)\psi_c^{-1}(v)dv.
\end{equation}
The argument is similar. It is enough to show that
\begin{equation}\label{1.58}
\int\limits_{F^{(r-c)k}\backslash \BA^{(r-c)k}}\int_{[V^c_0]}\varphi(vx_c(e))\psi_{c,0}^{-1}(v)\psi^{-1}(e)dvde=0
\end{equation}	
identically on $\pi$. The Fourier coefficient \eqref{1.58} corresponds to the following Whittaker pair $(\alpha,h)$:
$$
h(t)=\diag(t^{2ck},t^{2(c-1)k},...,t^2,I_{(r-c)k}),
$$
$$
\alpha=diag(J_k,...,J_k,J_{k+1},0),
$$
where $J_k$ is taken $c-1$ times. The Jordan form of $\alpha$ corresponds to the partition $(k+1,k^{c-1},1^{\ell-ck-1})$. Using Theorem \ref{thm 1.3}, we get a contradiction to the assumption that $\mathcal{O}(\pi)=(k^r)$. From \eqref{1.56}, \eqref{1.57}, we get	
\begin{equation}\label{1.59}
\mathcal{F}_{\psi_{\underline{p}}}(\varphi)(g)=\int\limits_{Y^1(\BA)}\cdots\int\limits_{Y^c(\BA)} \int_{[V^c]}\varphi(vy_c\cdots y_1w_0g)\psi_{c}^{-1}(v)dvdy_c\cdots dy_1.
\end{equation}
Note that $V^{r-1}=Z_\ell$, $\psi_{r-1}$ is the character \eqref{1.47} and \eqref{1.59}, for $c=r-1$, is	
\begin{equation}\label{1.60}
\mathcal{F}_{\psi_{\underline{p}}}(\varphi)(g)=\int\limits_{Y(\BA)} \mathcal{Z}_{\psi,k,r}(\varphi)(yw_0g)dy,
\end{equation}
where $Y=Y^{r-1}\cdots Y^1$. This is the subgroup of elements \eqref{1.51}, where the diagonal blocks are $I_k$ and those above the diagonal are zero. Since, by assumption, the l.h.s. of \eqref{1.60} is nontrivial on $\pi$, then 	
$\mathcal{Z}_{\psi,k,r}$ is nontrivial on $\pi$. This completes the proof of the theorem.	
	
\end{proof}

\begin{thm}\label{thm 1.7}
Let $\pi$ be an automorphic representation of $\GL_\ell(\BA)$, $\ell=kr$. Assume that $\mathcal{O}(\pi)=(k^r)$. Let $a\in \SL_r(\BA)$. Denote $\delta_k(a)=\diag(a,...,a)$ ($k$ times). Then, for all $\varphi\in \pi$, $g\in \GL_\ell(\BA)$, $a\in \SL_r(\BA)$,
\begin{equation}\label{1.61}
\mathcal{F}_{\psi_{(k^r)}}(\varphi)(\delta_k(a)g)=\mathcal{F}_{\psi_{(k^r)}}(\varphi)(g).
\end{equation}
\end{thm}
\begin{proof}
Note that the character $\psi_{(k^r)}$ is preserved under conjugation by $\delta_k(a)$. We may assume that $g=I_\ell$. Denote
$$
v(e)=\begin{pmatrix}1&e\\&I_{r-1}\end{pmatrix}.
$$
It is enough to prove \eqref{1.61} for $a=v(e)$, since these elements and $\SL_r(F)$ generate $\SL_r(\BA)$. We will prove that in the Fourier expansion of the function on $(F\backslash \BA)^{r-1}$,
$e\mapsto \mathcal{F}_{\psi_{(k^r)}}(\varphi)(\delta_k(v(e))$, only the trivial character contributes. It is enough to show that
\begin{equation}\label{1.62}
\int\limits_{(F\backslash \BA)^{r-1}}\mathcal{F}_{\psi_{(k^r)}}(\varphi)(\delta_k(v(e))\psi^{-1}(e_1)de=0,
\end{equation}
for all $\varphi\in \pi$. One can check that we can perform root exchange as follows. Let $Y_i$, $1\leq i\leq k-1$,  be the subgroup of $V_{r^k}$ consisting of the matrices
$$
\diag(I_{(i-1)r},\begin{pmatrix}1&0&0&0\\&I_{r-1}&y&0\\&&1&0\\&&&I_{r-1}\end{pmatrix},I_{(k-i-1)r}).
$$
Let $X_i$ be the subgroup consisting of the matrices
$$
\diag(I_{(i-1)r},v(x),I_{(k-i)r}).
$$
Then we can exchange in \eqref{1.62} $X_1$ and $Y_1$,...,$X_{k-1}$ and $Y_{k-1}$. The unipotent subgroup $V$, obtained after these root exchanges, is the one generated by $\{\diag(v(x_1),...,v(x_k))\}$ and the matrices
$$
\begin{pmatrix}I_r&X_{1,2}&\cdots&X_{1,k}\\&I_r&\cdots&X_{2,k}\\&&\ddots\\&&&I_r\end{pmatrix},
$$
where , for $1\leq i\leq k-1$, the first column of $X_{i,i+1}$ below the diagonal is zero. Then $V=V_{1,r^{k-1},r-1}$.
By \cite{GRS11}, Lemma 7.1, \eqref{1.62} holds identically iff the following Fourier coefficient is identically zero on $\pi$,
\begin{equation}\label{1.63}
\int\limits_{[V_{1,r^{k-1},r-1}]}\varphi(v)\eta_\psi^{-1}(v)dv,
\end{equation}
where, for
$$
v=\begin{pmatrix}1&x&\ast&\cdots&\ast&\ast\\&I_r&y_1&\cdots&\ast&\ast\\&&\ddots\\&&&I_r&y_{k-2}&\ast\\&&&&I_r&u\\&&&&&I_{r-1}\end{pmatrix}\in V_{1,r^{k-1},r-1}(\BA),
$$
$$
\eta_\psi(v)=\psi(x_r)\psi(tr(y_1+\cdots+y_{r-2}))\psi(u_{1,1}+u_{2,2}+\cdots+u_{r-1,r-1}+u_{r,1}).
$$
The Fourier coefficient \eqref{1.63} corresponds to the following Whittaker pair $(\alpha,h)$:
$$
h(t)=\diag(t^k,t^{k-2}I_r,...,t^{2-k}I_r,t^{-k}I_{r-1}),
$$
$$
\alpha=\begin{pmatrix}0_1\\\epsilon&0_r\\&I_r&0_r\\&&&\ddots\\&&&I_r&0_r\\&&&&A&0_{r-1}\end{pmatrix},
$$
where $0_i$ denotes the $i\times i$ zero matrix, $\epsilon$ is the column vector in $r$ coordinates $0,...,0,1$, and $A$ is the matrix whose first $r-1$ columns are $I_{r-1}$ and its last column has coordinates $1,0,...,0$. The Jordan form of $\alpha$ corresponds to the partition $(k+1,k^{r-2},k-1)$. Thus, by our assumption on $\pi$, \eqref{1.63} is identically zero on $\pi$ and hence so is \eqref{1.62}. This proves the theorem.

\end{proof}

\section{A descent operation on Eisenstein series induced from Speh representations}

In this section, we apply the global Bernstein-Zelevinsky derivative $\mathcal{D}_{\psi,kn}$ to an Eisenstein series, $E_{\Delta(\tau,\underline{m}),\underline{s}}$, attached to the parabolic induction
\begin{equation}\label{2.1}
\rho_{\Delta(\tau,\underline{m}),\underline{s}}=\Delta(\tau,m_1)|\cdot|^{s_1}\times\Delta(\tau,m_2)|\cdot|^{s_2}\times\cdots\times\Delta(\tau,m_k)|\cdot|^{s_k},
\end{equation}
where $\tau$ is an irreducible, unitary, automorphic, cuspidal representation of $\GL_n(\BA)$, $\underline{m}=(m_1,...,m_k)$ is a $k$-touple of positive integers, and $\underline{s}=(s_1,...,s_k)$ is a $k$-touple of complex numbers. Note that the representation \eqref{2.1} is of the group $\GL_{mn}(\BA)$, where $m=\sum_{i=1}^km_i$. We will prove that, as a representation of $\GL_{(m-k)n}(\BA)$, $\mathcal{D}_{\psi,kn}(E_{\Delta(\tau,\underline{m}),\underline{s}})$ is an Eisenstein series attached to $|\det\cdot|^{\frac{kn-1}{2}}\rho_{\Delta(\tau,\underline{m}-\underline{1}),\underline{s}}$, where $\underline{1}$ is the $k$-touple $(1,...,1)$. Here, we view $\GL_{(m-k)n}$, as the subgroup\\ $r_{(kn,1^{(m-k)n})}(\GL_{(m-k)n})$ in \eqref{1.38}, \eqref{1.40}. We think of $\mathcal{D}_{\psi,kn}$ as a descent map which associates to an automorphic representation of $\GL_{kn}(\BA)$, an automorphic representation of $\GL_{(m-k)n}(\BA)$.

\subsection{Relation of the $\psi_{(kn,1^{(m-k)n})}$-coefficients on $\rho_{\Delta(\tau,\underline{m}),\underline{s}}$ and derivatives}

By Theorem 1.1 in \cite{LX20}, for $Re(s_1)\gg \Re(s_2)\gg\dots\gg\Re(s_k)$,
\begin{equation}\label{2.2}
\mathcal{O}(E_{\Delta(\tau,\underline{m}),\underline{s}})=((kn)^{m_{i_1}},((k-1)n)^{m_{i_2}-m_{i_1}},...,(2n)^{m_{i_{k-2}}-m_{i_{k-1}}},n^{m_{i_{k-1}}-m_{i_k}}),
\end{equation}
where $m_{i_1}\leq\dots\leq  m_{i_k}$. The part of the proof where the r.h.s. of \eqref{2.2} bounds all elements of $\mathcal{O}(E_{\Delta(\tau,\underline{m}),\underline{s}})$ is relatively simple and is valid for all $\underline{s}$, where the Eisenstein series is holomorphic. It can be deduced using one finite place of $F$, where $\tau$ is unramified. This implies

\begin{prop}\label{prop 2.1}
	
Let $f_{\Delta(\tau,\underline{m}),\underline{s}}$ be a smooth, holomorphic section of $\rho_{\Delta(\tau,\underline{m}),\underline{s}}$, and let $E(f_{\Delta(\tau,\underline{m}),\underline{s}})(g)$ be the corresponding Eisenstein series. Then, for $g\in \GL_{mn}(\BA)$,
$$
\mathcal{D}^0_{\psi,kn}(E(f_{\Delta(\tau,\underline{m}),\underline{s}}))(g)=\mathcal{D}_{\psi,kn}(E(f_{\Delta(\tau,\underline{m}),\underline{s}}))(g).
$$
We conclude, from \eqref{1.34} and the notation there, when $kn$ is odd, that
\begin{equation}\label{2.3}
\mathcal{F}_{\psi_{(kn,1^{(m-k)n})}}(E(f_{\Delta(\tau,\underline{m}),\underline{s}}))(g)=\int_{Y(\BA)}\mathcal{D}_{\psi,kn}(E(f_{\Delta(\tau,\underline{m}),\underline{s}}))(ywg)dy,
\end{equation}
and, similarly, by \eqref{1.36}, when $kn$ is even,
\begin{equation}\label{2.4}
\mathcal{F}^1_{\psi_{(kn,1^{(m-k)n})}}(E(f_{\Delta(\tau,\underline{m}),\underline{s}}))(g)=\int_{Y(\BA)}\mathcal{D}_{\psi,kn}(E(f_{\Delta(\tau,\underline{m}),\underline{s}}))(ywg)dy.
\end{equation}

\end{prop}

\begin{proof}

Denote for short, $\varphi=E(f_{\Delta(\tau,\underline{m}),\underline{s}})$. Consider the following function on $\BA^{(m-k)n}$ (column vectors)
$$ 
u\mapsto \mathcal{D}^0_{\psi,kn}(\varphi)(\begin{pmatrix}I_{(m-k)n}&u\\&1\\&&I_{kn-1}\end{pmatrix}g)	
$$	
This is a smooth function on the compact abelian group $F^{(m-k)n}\backslash \BA^{(m-k)n}$. Consider its Fourier expansion. The contribution to the Fourier of each nontrivial character of $F^{(m-k)n}\backslash \BA^{(m-k)n}$ is zero. Indeed, all these characters form one orbit under the action of $\diag (GL_{(m-k)n}(F),I_{kn})$, hence it is enough to consider
$$
\int_{F^{(m-k)n}\backslash \BA^{(m-k)n}}\mathcal{D}^0_{\psi,kn}(\varphi)(\begin{pmatrix}I_{(m-k)n}&u\\&1\\&&I_{kn-1}\end{pmatrix}g)\psi^{-1}(u_{(m-k)n})du.
$$
This Fourier coefficient is equal to $\mathcal{D}^0_{\psi,kn+1}(\varphi)(g)$, which is identically zero by \eqref{2.2} and, by Prop. \ref{prop 1.4} (or Theorem \ref{thm 1.3} in this case), with $\ell=mn$ and $kn+1$ in place of $k$. Thus, only the trivial character contributes to the above Fourier expansion.	
	
\end{proof}

\subsection{The descent to $\GL_{(m-k)n}(\BA)$ of the Eisenstein series $E(f_{\Delta(\tau,\underline{m}),\underline{s}})$} 

Before we state the main theorem of this section, we need some notation. Denote by $w_0$ the following permutation matrix in $\GL_{mn}(F)$.
\begin{equation}\label{2.5}
w_0=(w^1_0,w^2_0)
\end{equation}
$$
w^1_0=\diag(A_1,A_2,...,A_k),\ \ \ \ A_i=\begin{pmatrix}I_{(m_i-1)n}\\0_{n\times(m_i-1)n}\end{pmatrix},
$$
$$
w^2_0=\diag(B_1,B_2,...,B_k),\ \ \ \ B_i=\begin{pmatrix}0_{(m_i-1)n\times n}\\I_n\end{pmatrix};
$$
$$
\epsilon_0=\begin{pmatrix}&&&I_n\\&&I_n\\&\dots\\I_n\end{pmatrix}_{kn\times kn};
$$
$$
\mathcal{B}=\left\{\begin{pmatrix}0&\\x_{2,1}&0\\x_{3,1}&x_{3,2}&0\\&&&\ddots\\x_{k-1,1}&x_{k-1,2}&&&0\\x_{k,1}&x_{k,2}&&\dots&x_{k,k-1}&0\end{pmatrix}\in M_{(m-k)n}: x_{i,j}\in M_{(m_i-1)n\times n}, i>j \right\}.
$$
Recall that, for $g\in GL_{(m-k)n}(\BA)$, $r_{(kn,1^{(m-k)n})}(g)=\diag(g,I_{kn})$. Denote, for short, 
$$
r(g)=r_{(kn,1^{(m-k)n})}(g).
$$

\begin{thm}\label{thm 2.2}
Let $f_{\Delta(\tau,\underline{m}),\underline{s}}$ be a smooth, holomorphic section of $\rho_{\Delta(\tau,\underline{m}),\underline{s}}$. Then, for $Re(s_i-s_{i+1})\gg 0$, $1\leq i\leq k-1$, $g\in \GL_{(m-k)n}(\BA)$,
\begin{equation}\label{2.6}
\mathcal{D}_{\psi,kn}(E(f_{\Delta(\tau,\underline{m}),\underline{s}}))(r(g))=\sum\limits_{\gamma\in P_{n(\underline{m}-\underline{1})}(F)\backslash \GL_{(m-k)n}(\BA)}\Lambda(f_{\Delta(\tau,\underline{m}),\underline{s}})(\gamma g),
\end{equation}
where
\begin{equation}\label{2.7}
\Lambda(f_{\Delta(\tau,\underline{m}),\underline{s}}) (g)=\int\limits_{V_{n^k}(\BA)}\int\limits_{\mathcal{B}(\BA)}(\otimes^k\mathcal{D}_{\psi,n})(f_{\Delta(\tau,\underline{m}),\underline{s}}(w_0\begin{pmatrix}g&bv\\&\epsilon_0v\end{pmatrix}))\psi^{-1}_{Z_{kn}}(v)dbdv.
\end{equation}
Here, $\otimes^k\mathcal{D}_{\psi,n}$ denotes the $k$-fold tensor product of the descents from $\Delta(\tau,m_i)$ on $\GL_{m_in}(\BA)$ to $\GL_{(m_i-1)n}(\BA)$, $1\leq i\leq k$.\\
The integral \eqref{2.7} converges absolutely for $Re(s_i-s_{i+1})\gg 0$, $1\leq i\leq k-1$, and continues to a meromorphic function in $\BC^k$. $\Lambda(f_{\Delta(\tau,\underline{m}),\underline{s}}) $ is a smooth meromorphic section of $|\det\cdot|^{\frac{kn-1}{2}}\rho_{\Delta(\tau,\underline{m}-\underline{1}),\underline{s}}$. Thus, $\mathcal{D}_{\psi,kn}(E(f_{\Delta(\tau,\underline{m}),\underline{s}}))\circ r$ is an Eisenstein series attached to $|\det\cdot|^{\frac{kn-1}{2}}\rho_{\Delta(\tau,\underline{m}-\underline{1}),\underline{s}}$.
\end{thm}

\begin{proof}

Denote for short, $\rho_{\underline{s}}=\rho_{\Delta(\tau,\underline{m}),\underline{s}}$ $f_{\underline{s}}=f_{\Delta(\tau,\underline{m}),\underline{s}}$. 
For $Re(s_i-s_{i+1})\gg 0$, $1\leq i\leq k-1$,\\
\\
$\mathcal{D}_{\psi,kn}(E(f_{\underline{s}}))(r(g))=$
\begin{equation}\label{2.8}
=\int\limits_{[V_{(m-k)n,1^{kn}}]}\sum\limits_w\sum\limits_{\gamma\in P^w(F)\backslash P_{(m-k)n,kn}(F)}f_{\underline{s}}(w\gamma vr(g))\psi_{V_{(m-k)n,1^{kn}}}^{-1}(v)dv.
\end{equation}
See \eqref{1.25}. Here, $P^w=P_{(m-k)n,kn}\cap w^{-1}P_{n\underline{m}}w$, and $w$ is summed over\\ $P_{n\underline{m}}(F)\backslash \GL_{mn}(F)/P_{(m-k)n,kn}(F)$. Since $\tau$ is cuspidal on $\GL_n(\BA)$, it suffices to consider only Weyl elements of $\GL_{mn}(F)$, which permute the blocks of the Levi subgroup $M_{n^m}(F)$. Thus, the following elements form a set of relevant representatives out of $P_{n\underline{m}}(F)\backslash \GL_{mn}(F)/P_{(m-k)n,kn}(F)$.
\begin{equation}\label{2.9}
w=w_{\underline{r}}=(\diag(w_1^1,...,w^1_k),\diag(w_1^2,...,w_k^2)),
\end{equation}
$$
w_i^1=\begin{pmatrix}I_{nr_i}\\0\end{pmatrix}\in M_{m_in\times r_in}(F),\ \ 
w_i^2=\begin{pmatrix}0\\I_{\ell_in}\end{pmatrix}\in M_{m_in\times\ell_in}(F),
$$
$$
r_i,\ell_u\geq 0,\ \ \ r_i+\ell_i=m_i,\ \ 1\leq i\leq k
$$
$$
r_1+\cdots+r_k=m-k.
$$
Note that $\ell_1+\cdots+\ell_k=k$.
The subgroup $P^{w_{\underline{r}}}$ consists of the elements of the form
\begin{equation}\label{2.10}
h=\begin{pmatrix}A^1&A^2\\0&A^3\end{pmatrix},
\end{equation}
$$
A^i=\begin{pmatrix}a^i_{1,1}&a^i_{1,2}&\dots&a^i_{1,k}\\&a^i_{2,2}&\dots&a^i_{2,k}\\&&\ddots\\&&&a^i_{k,k}\end{pmatrix},
$$
where, for $1\leq i\leq j\leq k$,
$$
a^1_{i,j}\in M_{r_in\times r_jn}(F),\ a^2_{i,j}\in M_{r_in\times \ell_jn}(F),\ \ 	a^3_{i,j}\in M_{\ell_in\times \ell_jn}(F).
$$
For the element \eqref{2.10}, we have
\begin{equation}\label {2.11}
whw^{-1}=\begin{pmatrix}B_{1,1}&B_{1,2}&\dots&B_{1,k}\\&B_{2,2}&\dots&B_{2,k}\\&&\ddots\\&&&B_{k,k}\end{pmatrix},
\end{equation}
where, for $1\leq i\leq j\leq k$,
$$
B_{i,j}=\begin{pmatrix}a_{i,j}^1&a^2_{i,j}\\&a^3_{i,j}\end{pmatrix}.
$$
We will show that only $w_{\underline{r}}$, with $\ell_1=\cdots=\ell_k=1$ (i.e. $r_i=m_i-1$, $1\leq i\leq k$) contributes to \eqref{2.8}. The contribution of $w_{\underline{r}}$ to \eqref{2.8} is
\begin{equation}\label{2.12}
\int\limits_{[V_{(m-k)n,1^{kn}}]}\sum\limits_{\epsilon'} \sum\limits_{\eta\in P^{\underline{r},\epsilon'}(F)\backslash P_{(m-k)n,1^{kn}}(F)} f_{\underline{s}}(w_{\underline{r}}\epsilon'\eta vr(g))\psi_{V_{(m-k)n,1^{kn}}}^{-1}(v)dv,
\end{equation}
where $\epsilon'$ runs over
$$
P^{w_{\underline{r}}}(F)\backslash P_{(m-k)n,kn}(F)/P_{(m-k)n,1^{kn}}(F)\cong P_{n\underline{\ell}}(F)\backslash \GL_{kn}(F)/V_{1^{kn}}(F),
$$
and $P^{\underline{r},\epsilon'}=P_{(m-k)n,1^{kn}}\cap (\epsilon')^{-1}P^{w_{\underline{r}}} \epsilon'$. We may take $\epsilon'=\diag(I_{(m-k)n},\epsilon)$, where $\epsilon$ is in the Weyl group of $\GL_{kn}$, modulo the Weyl group of $M_{n\underline{\ell}}$ from the left. Again, due to the cuspidality of $\tau$, we may assume that $\epsilon$ permutes the blocks of $M_{n^k}$. The subgroup $P^{\underline{r},\epsilon'}$ cosists of the elements of the form
\begin{equation}\label{2.13}
h=\begin{pmatrix}A&C\\0&z\end{pmatrix},
\end{equation}where $A\in P_{n\underline{r}}$ (as $A^1$ in \eqref{2.10}), $C\epsilon^{-1}$ has the form of $A^2$ in \eqref{2.10}, and $z\in V_{1^{kn}}\cap \epsilon^{-1}P_{n\underline{\ell}}\epsilon$. Denote
$$
V^{\underline{r},\epsilon}=V_{(m-k)n,1^{kn}}\cap (\epsilon')^{-1}P^{w_{\underline{r}}}\epsilon'.
$$
The elements of this subgroup have the form \eqref{2.13} with $A=I_{(m-k)n}$. For a given $\underline{r}$, the contribution of $\epsilon$ to \eqref{2.12} is 
\begin{equation}\label{2.14}
\sum\limits_{\gamma\in P_{n\underline{r}}(F)\backslash \GL_{(m-k)n}(F)}\int\limits_{V^{\underline{r},\epsilon}(\BA)\backslash V_{(m-k)n,1^{kn}}(\BA)}\int\limits_{[V^{\underline{r},\epsilon}]} f_{\underline{s}}(w_{\underline{r}}\epsilon'uvr(\gamma g))\psi_{V_{(m-k)n,1^{kn}}}^{-1}(v)dudv.
\end{equation}
Let $\sigma$ be the permutation of $\{1,...,k\}$, such that if $z_i\in V_{1^{n\ell_{\sigma(i)in}}}(\BA)$, $1\leq i\leq k$, then $\epsilon\diag(z_1,...,z_k)\epsilon^{-1}=\diag(z_{\sigma^{-1}(1)},...,z_{\sigma^{-1}(k)})$. Consider the subgroup $V^{\underline{r},\epsilon}_0(\BA)$ of elements of the form
\begin{equation}\label{2.15}
\begin{pmatrix}I_{(m-k)n}&A^2\epsilon\\0&z\end{pmatrix},
\end{equation}
where $z=diag(z_1,...,z_k)$, $z_i\in V_{1^{n\ell_{\sigma(i)in}}}(\BA)$, $1\leq i\leq k$, and $A^2$ is as in \eqref{2.10} with adele coordinates. By \eqref{2.11}, conjugation of \eqref{2.15} by $w_{\underline{r}}$, will result with an inner integration to the $du$-integration in \eqref{2.14}, with the application to $\Delta(\tau,m_i)$, $1\leq i\leq k$, of the constant term along $V_{nr_i,n\ell_i}$, followed by the Whittaker coefficient for the block $\GL_{\ell_in}$, that is $\diag(I_{r_in},z)\mapsto \psi_{Z_{\ell_in}}(z)$, $z\in V_{1^{\ell_in}}(\BA)$. The last constant term takes $\Delta(\tau,m_i)$ to $\Delta(\tau,r_i)|\cdot|^{-\frac{\ell_i}{2}}\otimes \Delta(\tau,\ell_i)|\cdot|^{\frac{r_i}{2}}$. Since $\Delta(\tau,\ell_i)$ is not generic, for $\ell_i\geq 2$, we conclude that if the contribution of $w_{\underline{r}}$ to \eqref{2.8} is nonzero, then we must have $\ell_i\leq 1$, for all $1\leq i\leq k$. By \eqref{2.9}, this means that $\ell_1=\cdots=\ell_k=1$. 

Note that for $\underline{r}^0=(m_1-1,...,m_k-1)$, $w_{\underline{r}^0}=w_0$ (see \eqref{2.5}). For this element, let us show that only $\epsilon=\epsilon_0$ contributes to \eqref{2.14}.
We assume that $\epsilon$ is a permutation matrix, permuting the blocks of $M_{n^k}$. Thus, $\epsilon$ is a $k\times k$ matrix of $n\times n$ blocks which are zero or $I_n$. Assume that the first block row of $\epsilon$ has the form 
$(0,...,I_n,...,0)$, where $I_n$ stands in the $j$-th block, with $j<k$. Then we can find $z_0\in Z_{nk}(\BA)$, such that $\diag(I_{(m-k)n},z_0) \in (\epsilon')^{-1}w^{-1}_0V_{n\underline{m}}(\BA)w_0\epsilon'$, with $\psi_{V_{(m-k)n,1^{kn}}}(\diag(I_{(m-k)n},z_0))\neq 1$. For this, take $z_0$, such that, for $1\leq i\neq j\leq k$, the $i$-th block row is the $i$-th block row of $I_{nk}$, and the $j$-th block row of $z_0$ has the form $(0,...,I_n,x,...)$, where $I_n$ stands in the $j$-th block. Then $\psi_{V_{(m-k)n,1^{kn}}}(\diag(I_{(m-k)n},z_0))=\psi_{Z_{nk}}(z_0)=\psi(x_{n,1})$. We conclude that we must have 
$$
\epsilon=\begin{pmatrix}0&I_n\\\epsilon_1&0\end{pmatrix}.
$$
Repeating the argument, we get that we must have $\epsilon=\epsilon_0$. Thus, \eqref{2.8} is equal to \eqref{2.14}, with $\underline{r}=\underline{r}^0$ and $\epsilon=\epsilon_0$, and then it is easy to see that \eqref{2.14} is equal to the r.h.s. of \eqref{2.6}.

Let us consider \eqref{2.7} (for $Re(s_i-s_{i+1})\gg 0$, $1\leq i\leq k-1$). We can write it as the composition of the following three operators on $f_{\underline{s}}$: first, we take on each inducing factor $\Delta(\tau,m_i)$, the constant term along $V_{(m_i-1)n,n}$. This takes $f_{\underline{s}}$ to $\tilde{f}_{\underline{s}}$, which lies in the parabolic induction
\begin{equation}\label{2.16}
\Delta(\tau,m_1-1)|\cdot|^{s_1-\frac{1}{2}}\times\tau|\cdot|^{s_1+\frac{m_1-1}{2}}\times\cdots\times \Delta(\tau,m_k-1)|\cdot|^{s_k-\frac{1}{2}}\times\tau|\cdot|^{s_k+\frac{m_k-1}{2}}.
\end{equation}
Then we apply the intertwining operator
\begin{equation}\label{2.17}
M(\tilde{f}_{\underline{s}})(h)=\int\limits_{\mathcal{B}(\BA)}\tilde{f}_{\underline{s}}(w_0\begin{pmatrix}(I_{(m-k)n}&b\\&I_{kn}\end{pmatrix})h)db.
\end{equation}
This operator has a meromorphic continuation to $\BC^k$. It takes the paraolic induction \eqref{2.16} to
\begin{equation}\label{2.18}
\Delta(\tau,m_1-1)|\cdot|^{s_1-\frac{1}{2}}\times\cdots\times \Delta(\tau,m_k-1)|\cdot|^{s_k-\frac{1}{2}}\times \tau|\cdot|^{s_1+\frac{m_1-1}{2}}\times\cdots\times\tau|\cdot|^{s_k+\frac{m_k-1}{2}}
\end{equation}
$$
\cong \rho_{\Delta(\tau,\underline{m}-\underline{1}),\underline{s}-\frac{1}{2}\cdot \underline{1}}\times (\tau|\cdot|^{s_1+\frac{m_1-1}{2}}\times\cdots\times\tau|\cdot|^{s_k+\frac{m_k-1}{2}}).
$$
The third operator, applied to \eqref{2.17} is
\begin{equation}\label{2.19}
id_{\rho_{\Delta(\tau,\underline{m}-\underline{1}),\underline{s}-\frac{1}{2}\cdot \underline{1}}}\otimes \mathcal{W}_{\psi_{Z_{nk}}},
\end{equation}
where $\mathcal{W}_{\psi_{Z_{nk}}}$ is the Jacquet integral, giving the $\psi_{Z_{nk}}$-Whittaker functional for $\tau|\cdot|^{s_1+\frac{m_1-1}{2}}\times\cdots\times\tau|\cdot|^{s_k+\frac{m_k-1}{2}}$. For $\varphi_{\underline{s}+\frac{1}{2}\cdot(\underline{m}-\underline{1})}$ in $\tau|\cdot|^{s_1+\frac{m_1-1}{2}}\times\cdots\times\tau|\cdot|^{s_k+\frac{m_k-1}{2}}$,\\
\\
$\mathcal{W}_{\psi_{Z_{nk}}}(\varphi_{\underline{s}+\frac{1}{2}\cdot(\underline{m}-\underline{1})})=$
$$
=\int\limits_{V_{n^k}(\BA)}\int\limits_{[Z_n]^k}\varphi_{\underline{s}+\frac{1}{2}\cdot(\underline{m}-\underline{1})}(\begin{pmatrix}z_1&\\&\ddots\\&&z_k\end{pmatrix}\epsilon_0v)\psi^{-1}_{Z_n}(z_1\cdots z_k)\psi^{-1}_{Z_{kn}}(v)d(z_1,...,z_k)dv.
$$
This is meromorphic in $\BC^n$. All in all, 
\begin{equation}\label{2.20}
\Lambda(f_{\underline{s}})=(id_{\rho_{\Delta(\tau,\underline{m}-\underline{1}),\underline{s}-\frac{1}{2}\cdot \underline{1}}}\otimes \mathcal{W}_{\psi_{Z_{nk}}})\circ M(\tilde{f}_{\underline{s}})\Big|_{r(\GL_{(m-k)n}(\BA))}
\end{equation}
is a smooth meromorphic section of $|\det\cdot|^{\frac{kn-1}{2}}\rho_{\Delta(\tau,\underline{m}-\underline{1}),\underline{s}}$.
This cocludes the proof of the theorem.

\end{proof}

\begin{rmk}\label{rmk 2.1}
Theorem \ref{thm 2.2} is valid with a similar proof when we take an Eisenstein series attached to $\Delta(\tau_1,m_1)|\cdot|^{s_1}\times\cdots\times\Delta(\tau_k,m_k)|\cdot|^{s_k}$, where $\tau_i$, $1\leq i\leq k$, is an irreducible, cuspidal, automorphic representation of $\GL_{n_i}(\BA)$. Then the application of $\mathcal{D}_{\psi,n_1+\cdots+n_k}$ gives an Eisenstein series attached to 	$\Delta(\tau_1,m_1-1)|\cdot|^{s_1}\times\cdots\times\Delta(\tau_k,m_k-1)|\cdot|^{s_k}$.	
\end{rmk}
\begin{rmk}\label{rmk 2.2}
As we pointed out in the introduction, Theorem \ref{thm 2.2} is a global analog of the calculation of the Bernstein-Zelevinsky derivative of order $kn$ of the parabolic induction of $k$ segment representations. See Lemma 4.5 in \cite{BZ77} and Theorem 3.5 in \cite{Ze80}. (The same holds for the more general form as in Remark \ref{rmk 2.1}.)
\end{rmk}

\subsection{The section $\Lambda(f_{\Delta(\tau,\underline{m}),\underline{s}})$}

We specify a domain of holomorphy of the section $\Lambda(f_{\Delta(\tau,\underline{m}),\underline{s}})$. 

\begin{thm}\label{thm 2.3}

The section $\Lambda(f_{\Delta(\tau,\underline{m}),\underline{s}})$ is holomorphic at $\underline{s}$, such that
\begin{equation}\label{2.21}
Re(s_i-s_j)\geq\frac{m_j-m_i}{2},\ \ \forall 1\leq i<j\leq k.
\end{equation}

\end{thm}

\begin{proof}
	
Let us start with the intertwining operator \eqref{2.17}. We can write it as the composition of elementary intertwining operators, described according to the 
following simple permutations on the inducing data of \eqref{2.16}. The first elementary intertwining operator flips
$$
\tau|\cdot|^{s_1+\frac{m_1-1}{2}}\times \Delta(\tau,m_2-1)|\cdot|^{s_2-\frac{1}{2}}\rightarrow \Delta(\tau,m_2-1)|\cdot|^{s_2-\frac{1}{2}}\times \tau|\cdot|^{s_1+\frac{m_1-1}{2}}.
$$
Then we carry the next two flips, first for $i=2$, then for $i=1$,
$$
\tau|\cdot|^{s_i+\frac{m_i-1}{2}}\times \Delta(\tau,m_3-1)|\cdot|^{s_3-\frac{1}{2}}\rightarrow \Delta(\tau,m_3-1)|\cdot|^{s_3-\frac{1}{2}}\times \tau|\cdot|^{s_i+\frac{m_i-1}{2}}.
$$
At stage $j-1$, $2\leq j\leq k$, flip, for $i=j-1,j-2,...,1$, in this order,	
$$
\tau|\cdot|^{s_i+\frac{m_i-1}{2}}\times \Delta(\tau,m_j-1)|\cdot|^{s_j-\frac{1}{2}}\rightarrow \Delta(\tau,m_j-1)|\cdot|^{s_j-\frac{1}{2}}\times \tau|\cdot|^{s_i+\frac{m_i-1}{2}}.
$$
Let $S$ be a finite set of places, containing the archimedean places, outside which $\tau$ is unramified. We may assume that $\tilde{f}_{\underline{s}}$ is decomposable. Let $v\notin S$, and let $\tilde{f}^0_{\underline{s},v}$, $\varphi^0_{\underline{s},v}$ be the unramified vectors in \eqref{2.16}, \eqref{2.18}, at $v$, normalized such that
$$
\tilde{f}^0_{\underline{s},v}(I_{mn})=\xi_{1,v}^0\otimes e_v^0\otimes\cdots \xi_{k,v}^0\otimes e_v^0,
$$
$$
\varphi^0_{\underline{s},v}(I_{mn})=\xi^0_{1,v}\otimes\cdots\otimes \xi^0_{k,v}\otimes e^0_v\otimes\cdots e^0_v,
$$
where $\xi^0_{i,v}$ are unramified vectors for the factor at $v$, $\Delta(\tau_v,m_i-1)$, of $\Delta(\tau,m_i-1)$, $1\leq i\leq k$, and $e_v^0$ is an unramified vector for $\tau_v$. Then
$$
M_v(\tilde{f}^0_{\underline{s},v})=\prod\limits_{1\leq i<j\leq k}\frac{L(\tau_v\times \Delta(\hat{\tau}_v,m_j-1),s_i-s_j+\frac{m_i}{2})}{L(\tau_v\times \Delta(\hat{\tau}_v,m_j-1),s_i-s_j+\frac{m_i}{2}+1)}\varphi^0_{\underline{s},v}=
$$
$$
=\prod\limits_{1\leq i<j\leq k}\frac{L(\tau_v\times \hat{\tau}_v,s_i-s_j+\frac{m_i-m_j+2}{2})}{L(\tau_v\times \hat{\tau}_v,s_i-s_j+\frac{m_i+m_j}{2})}\varphi^0_{\underline{s},v}.
$$		
Write
\begin{equation}\label{2.22}
M(\tilde{f}_{\underline{s}})=\prod\limits_{1\leq i<j\leq k}\frac{L(\tau\times \hat{\tau},s_i-s_j+\frac{m_i-m_j+2}{2})}{L(\tau\times \hat{\tau},s_i-s_j+\frac{m_i+m_j}{2})}(\otimes_vN_v(\tilde{f}_{\underline{s},v})),
\end{equation}
where $N_v$ is the local normalized intertwining operator (we dropped the corresponding epsilon factors). At almost all places $v\notin S$, we have $N_v(\tilde{f}^0_{\underline{s},v})=\varphi^0_{\underline{s}}$. Since $\tau_v$ and $\Delta(\tau_v,m_i-1)$ are unitary, we know from \cite{MW89}, I.10, that $N_v(\tilde{f}_{\underline{s},v})$ is holomorphic when
\begin{equation}\label{2.23} 
Re(s_i-s_j)\geq -\frac{m_i}{2},\ \ \forall 1\leq i<j\leq k.
\end{equation} 
In order to obtain $\Lambda(f_{\underline{s}})$, it remains to apply \eqref{2.19} to \eqref{2.22}. Consider a decomposable smooth, holomorphic section $\xi_{\underline{s}+\frac{1}{2}\cdot (\underline{m}-\underline{1})}$ of $\tau|\cdot|^{s_1+\frac{m_1-1}{2}}\times\cdots\times \tau|\cdot|^{s_k+\frac{m_k-1}{2}}$. Assume that, for all $v$, $\tau_v$ is realized in is $\psi_v$-Whittaker model $W(\tau_v,\psi_v)$. Assume also that, for $v\notin S$, $\xi_{\underline{s}+\frac{1}{2}\cdot (\underline{m}-\underline{1}),v}(I_{kn})=W^0_{\psi_v}\otimes\cdots W^0_{\psi_v}$, $k$ times, where $W^0_{\psi_v}$ is the unramified Whittaker function in $W(\tau_v,\psi_v)$, such that $W^0_{\psi_v}(I_n)=1$. Then by the Casselman-Shalika formula,\\
\\
$\mathcal{W}_{\psi_{Z_{kn}}}(\xi_{\underline{s}+\frac{1}{2}\cdot (\underline{m}-\underline{1})})=$
\begin{equation}\label{2.24}
=\prod\limits_{1\leq i<j\leq k}\frac{L_S(\tau\times\hat{\tau},s_i-s_j+\frac{m_i-m_j+2}{2})}{L(\tau\times\hat{\tau},s_i-s_j+\frac{m_i-m_j+2}{2})}\prod_{v\in S}\mathcal{W}_{\psi_{Z_{kn},v}}(\xi_{\underline{s}+\frac{1}{2}\cdot (\underline{m}-\underline{1}),v}),
\end{equation}
where $L_S(\tau\times\hat{\tau},s_i-s_j+\frac{m_i-m_j+2}{2})=\prod_{v\in S}L(\tau_v\times\hat{\tau}_v,s_i-s_j+\frac{m_i-m_j+2}{2})$, and for $v\in S$, 
$$
\mathcal{W}_{\psi_{Z_{kn},v}}(\xi_{\underline{s}+\frac{1}{2}\cdot (\underline{m}-\underline{1}),v})=\int\limits_{V_{n^k}(F_v)}\xi_{\underline{s}+\frac{1}{2}\cdot (\underline{m}-\underline{1}),v}(\epsilon_0v)\psi_{Z_{kn},v}^{-1}(v)dv.
$$
This is the local Jacquet integral, which we know is holomorphic. From \eqref{2.22}, \eqref{2.24}, we get that when $\underline{s}$ satisfies \eqref{2.23}, then for $\Lambda(f_{\underline{s}})(I_{(m-k)n})$ to be holomorphic, it suffices to require that $\frac{L_S(\tau\times\hat{\tau},s_i-s_j+\frac{m_i-m_j+2}{2})}{L(\tau\times \hat{\tau},s_i-s_j+\frac{m_i+m_j}{2})}$ is holomorphic, for all $1\leq i<j\leq k$.
Since $\tau_v$ is unitary and generic, for all $v\in S$, the last function is holomorphic when $Re(s_i-s_j)+\frac{m_i-m_j}{2}\geq 0$. This proves the theorem.	
\end{proof}

Let us write the section $\Lambda(f_{\underline{s}})$, as a decomposable section, for decomposable $f_{\underline{s}}$. Indeed, in \eqref{2.20}, the intertwining operator $M$ is decomposable and so is the $\psi_{Z_{nk}}$-Whittaker functional. Next, recall that $\tilde{f}_{\underline{s}}$ is obtained from $f_{\underline{s}}$ by applying to each inducing factor $\Delta(\tau,m_i)$ the constant term along $V_{(m_i-1)n,n}$. This produces, at each place $v$, a unique model of $\Delta(\tau_v,m_i)$, as follows. Fix a realization space for $\Delta(\tau_v,m_i-1)$, which we keep denoting by $\Delta(\tau_v,m_i-1)$, and realize $\tau_v$ in its $\psi_v$-standard Whittaker model. Now realize $\Delta(\tau_v,m_i)$ as the unique irreducible subrepresentation of $\Delta(\tau_v,m_i-1)|\cdot|^{-\frac{1}{2}}\times \mathcal{W}(\tau_v,\psi_v)|\cdot|^{\frac{m_i-1}{2}}$. Thus, there are embeddings
$\alpha_{i,v}: \Delta(\tau_v,m_i)\rightarrow\Delta(\tau_v,m_i-1)|\cdot|^{-\frac{1}{2}}\times \mathcal{W}(\tau_v,\psi_v)|\cdot|^{\frac{m_i-1}{2}}$, such that, for a decomposable section $f_{\underline{s}}$,
\begin{equation}\label{2.25}
\Lambda(f_{\underline{s}})=\otimes_v(id_{\rho_{\Delta(\tau_v,\underline{m}-\underline{1}),\underline{s}-\frac{1}{2}\cdot \underline{1}}}\otimes \mathcal{W}_{\psi_{Z_{nk},v}})\circ M_v((\alpha_v\circ f_{\underline{s},v})\Big|_{r(\GL_{(m-k)n}(F_v))}),
\end{equation}
where, $\alpha_v=\otimes_{i=1}^k\alpha_{i,v}$, and $\mathcal{W}_{\psi_{Z_{nk},v}}$ denotes a map assigning to $\xi_v$ in 
$\eta_v=\tau_v|\cdot|^{s_1+\frac{m_1-1}{2}}\times\cdots\times \tau_v|\cdot|^{s_k+\frac{m_k-1}{2}}$, the Whittaker function $h\mapsto W_v(\eta_v(h)\xi_v)$, where $W_v$ is a $\psi_{Z_{nk},v}$-Whittaker functional on $\eta_v$. It is unique up to scalars. Denote the $v$-component of \eqref{2.25} by $\Lambda_v(f_{\underline{s},v})$.

\begin{prop}\label{prop 2.4}
Fix a place $v$ and fix $\underline{s}^0\in \BC^k$ satisfying \eqref{2.21}. Consider the local sections $\Lambda_v(f_{\underline{s},v})$ at the point $\underline{s}^0$. View these as maps to $\rho_{\Delta(\tau_v,\underline{m}-\underline{1}),\underline{s}^0-\frac{1}{2}\cdot \underline{1}}$. Then, when $v$ is finite, these are surjective maps. When $v$ is Archimedean, their image is dense (in the Frechet topology).
\end{prop}

\begin{proof}
The proof is similar to that of Prop. 2.6 in \cite{GS22}.
Denote by $\Lambda_{v,\underline{s}^0}$ be the following linear map on $V_{\rho_{\Delta(\tau_v,\underline{m}),\underline{s}^0}}$. Let $f_0$ be a function in $V_{\rho_{\Delta(\tau_v,\underline{m}),\underline{s}^0}}$, and let  $f_{\Delta(\tau_v,\underline{m}),\underline{s}}$ be any smooth, holomorphic section of $\rho_{\Delta(\tau_v,\underline{m}),\underline{s}}$, such that $f_{\Delta(\tau_v,\underline{m}),\underline{s}^0}=f_0$. Then
$$
\Lambda_{v,\underline{s}^0}(f_0)=\Lambda(f_{\Delta(\tau_v,\underline{m}),\underline{s}})_{\large |_{\underline{s}=\underline{s}^0}}.
$$
This is well defined, since $\Lambda(f_{\Delta(\tau_v,\underline{m}),\underline{s}})_{\large |_{\underline{s}=\underline{s}^0}}$ depends only on $f_{\Delta(\tau_v,\underline{m}),\underline{s}^0}$. Let $\xi_0$ be an element in the dual of $\rho_{\Delta(\tau_v,\underline{m}),\underline{s}^0}$, realized as $\rho_{\Delta(\hat{\tau}_v,\underline{m}),-\underline{s}^0}$. Assume that it is zero on the image of $\Lambda_{v,\underline{s}^0}$. Then we want to prove that $\xi_0=0$.
Our assumption is that for all $f_0$,
\begin{equation}\label{2.26}
<\Lambda_{v,s_0}(f_0),\xi_0>=0.
\end{equation}
Let
$$
w'_0=\diag(w_0'',I_{kn}), \  \ w_0''=\begin{pmatrix}&&&I_{(m_1-1)n}\\&&\cdot\\&\cdot\\I_{(m_k-1)n}\end{pmatrix}.
$$
Let us take in \eqref{2.26}, $f_0=\rho((w_0\epsilon'_0w_0')^{-1})f'_0$, where $f'_0$ is supported in $P_{n\underline{m}}(F_v)\bar{V}_{n\underline{m}}(F_v)$, and the support in $\bar{V}_{n\underline{m}}(F_v)$ is compact. Then we can write the l.h.s. of \eqref{2.26} as an absolutely convergent integral. This will be clear from what follows.
\begin{equation}\label{2.27}
<\Lambda_{v,s_0}(f_0),\xi_0>=\int\limits_{V_{(m_k-1)n,...,(m_1-1)n}(F_v)}<\Lambda_{v,\underline{s}^0}(f_0)(w_0''u),\xi_0(w''_0u)>du,
\end{equation}
where the inner pairing in \eqref{2.27} is the pairing between $\otimes_{i=1}^k\Delta(\tau_v,m_i-1)$ and $\otimes_{i=1}^k\Delta(\hat{\tau}_v,m_i-1)$. Write $u\in V_{(m_k-1)n,...,(m_1-1)n}(F_v)$ as
$$
u=u(y)=\begin{pmatrix}I_{(m_k-1)n}&y_{k,k-1}&\cdots&y_{k,1}\\&I_{(m_{k-1}-1)n}&\cdots&y_{k-1,1}\\&&\ddots\\&&&I_{(m_1-1)n}\end{pmatrix}.
$$
Then \eqref{2.27} is equal to
\begin{equation}\label{2.28}
\int <(\alpha'_v\circ f'_0)(\begin{pmatrix}I_{m_1n}\\c_{2,1}(x,y,z)&I_{m_2n}\\&&\ddots\\c_{k,1}(x,y,z)&c_{k,2}(x,y,z)&\cdots&I_{m_kn}\end{pmatrix}),\xi_0(w''_0u(y))>\chi_\psi^{-1}(z)dxdzdy.
\end{equation}
Here, for $1\leq j<i\leq k$,
$$
c_{i,j}(x,y,z)=\begin{pmatrix}y_{i,j}&x_{i,j}\\0&z_{i,j}\end{pmatrix}\in M_{m_in\times m_jn}(F_v),\ \ y_{i,j}\in M_{(m_i-1)n\times (m_j-1)n}(F_v), 
$$
$$
x_{i,j}\in M_{(m_i-1)n\times n}(F_v),\ \ z_{i,j}\in M_{n\times n}(F_v);
$$
$$ 
\chi_\psi(z)=\psi(\sum\limits_{i=1}^{k-1}((z_{i,i+1})_{n,1}). 
$$
Finally, $\alpha'_v$ is obtained from $\alpha_v$, by applying, for $1\leq i\leq k$, $\alpha_{i,v}$ to $\Delta(\tau_v,m_i)$ and evaluating at $I_n$ at the $W(\tau_v,\psi_v)$-factor. Since \eqref{2.28} is zero, for all $f'_0$ in $V_{\rho_{\Delta(\tau_v,\underline{m}),\underline{s}^0}}$, supported in $P_{n\underline{m}}(F_v)\bar{V}_{n\underline{m}}(F_v)$, compactly modulo $P_{n\underline{m}}(F_v)$, we conclude that $\xi_0$ is zero on the open cell $P_{n(\underline{m}-\underline{1}}(F_v)w''_0V_{(m_k-1)n,...,(m_1-1)n}(F_v)$, and hence $\xi_0=0$. This proves the proposition.

\end{proof}

\subsection{The case $k=2$}. From now on, we concentrate on the case $k=2$. So, we consider $\underline{m}=(m_1,m_2)$, $m=m_1+m_2$, $\underline{s}=(s_1,s_2)$, the parabolic induction $\rho_{\Delta(\tau,\underline{m}),\underline{s}}=\Delta(\tau,m_1)|\cdot|^{s_1}\times \Delta(\tau,m_2)|\cdot|^{s_2}$, and a corresponding Eisenstein series $E(f_{\Delta(\tau,\underline{m}),\underline{s}})$. We saw in Theorem \ref{thm 2.2} that $\mathcal{D}_{\psi,2n}(E(f_{\Delta(\tau,\underline{m}),\underline{s}}))\circ r$ is an Eisenstein series on $\GL_{(m-2)n}(\BA)$, attached to the section $\Lambda(f_{\Delta(\tau,\underline{m}),\underline{s}})$ of $|\det\cdot|^{\frac{2n-1}{2}}(\Delta(\tau,m_1-1)|\cdot|^{s_1}\times \Delta(\tau,m_2-1)|\cdot|^{s_2})$. We saw in Theorem \ref{thm 2.3} that $\Lambda(f_{\Delta(\tau,\underline{m}),\underline{s}})$ is holomorphic at $(s_1,s_2)$, such that
\begin{equation}\label{2.29}
Re(s_1-s_2)\geq \frac{m_2-m_1}{2}.
\end{equation}
We may assume that $m_1\geq m_2$. Let us explain why we can make this assumption. Indeed, let us show that we have an analogous picture, namely another descent operation, taking $E(f_{\Delta(\tau,\underline{m}),\underline{s}})$ to an Eisenstein series on $\GL_{(m-2)n}(\BA)$, attached to $|\det\cdot|^{\frac{1-2n}{2}}\Delta(\tau,m_1-1)|\cdot|^{s_1}\times \Delta(\tau,m_2-1)|\cdot|^{s_2}$, corresponding to a section, which is holomorphic
when $Re(s_1-s_2)\geq \frac{m_1-m_2}{2}$. For this, we will use the following analog of the Bernstein-Zelevinski derivative, related to \eqref{1.41}. Let $g\in \GL_{(m-2)n}(\BA)$. Then (see \eqref{1.41}),
\begin{equation}\label{2.30}
\mathcal{D}'_{\psi,2n}(E(f_{\Delta(\tau,\underline{m}),\underline{s}}))(g)=
\int\limits_{[V_{1^{2n},(m-2)n}]}E(f_{\Delta(\tau,\underline{m}),\underline{s}})(vg)\psi^{-1}_{V_{1^{2n},(m-2)n}}(v)dv.
\end{equation}.
Using Prop. \ref{prop 1.5}, we have the following analog of Prop. \ref{prop 2.1}. (It is, of course, valid for any $k$). 

\begin{prop}\label{prop 2.5}
	
We have, for $g\in \GL_{mn}(\BA)$,
$$
{}^0\mathcal{D}'_{\psi,2n}(E(f_{\Delta(\tau,\underline{m}),\underline{s}}))(g)=\mathcal{D}'_{\psi,2n}(E(f_{\Delta(\tau,\underline{m}),\underline{s}}))(g).
$$
We conclude from \eqref{1.43} that
\begin{equation}\label{2.31}
\mathcal{F}^1_{\psi_{(2n,1^{(m-2)n})}}(E(f_{\Delta(\tau,\underline{m}),\underline{s}}))(g)=\int_{Y'(\BA)}\mathcal{D}'_{\psi,2n}(E(f_{\Delta(\tau,\underline{m}),\underline{s}}))(ywg)dy.
\end{equation}
	
\end{prop}

Let us write down the analog of Theorem \ref{thm 2.2} for $k=2$. (This can be done for any $k$.) Let
$$
\tilde{w}_0=\begin{pmatrix}I_n&0&0&0\\0&0&I_{(m_1-1)n}&0\\0&I_n&0&0\\0&0&0&I_{(m_2-1)n}\end{pmatrix},
$$
$$
\mathcal{B}'=\left\{\begin{pmatrix}0&0\\x&0\end{pmatrix}\in M_{2n\times(m-2)n}:x\in M_{n\times (m_1-1)n}\right\}.
$$
Denote, for $g\in \GL_{(m-2)n}(\BA)$,
$$
r'(g)=\diag(I_{2n},g).
$$
\begin{thm}\label{thm 2.6}
For $Re(s_1-s_2)\gg 0$, $g\in \GL_{(m-2)n}(\BA)$,
\begin{equation}\label{2.32}
\mathcal{D}'_{\psi,2n}(E(f_{\Delta(\tau,\underline{m}),\underline{s}}))(r'(g))=\sum\limits_{\gamma\in P_{n((m_1-1),(m_2-1))}(F)\backslash \GL_{(m-2)n}(\BA)}\Lambda'(f_{\Delta(\tau,\underline{m}),\underline{s}})(\gamma g),
	\end{equation}
	where
\begin{equation}\label{2.33}
\Lambda'(f_{\Delta(\tau,\underline{m}),\underline{s}}) (g)=\int\limits_{V_{n^2}(\BA)}\int\limits_{\mathcal{B}'(\BA)}(\otimes^2\mathcal{D}'_{\psi,n})(f_{\Delta(\tau,\underline{m}),\underline{s}}(\tilde{w}_0\begin{pmatrix}\epsilon_0v&b\\&g\end{pmatrix}))\psi^{-1}_{Z_{2n}}(v)dbdv.
\end{equation}
The integral \eqref{2.33} converges absolutely for $Re(s_1-s_2)\gg 0$, and continues to a meromorphic function in $\BC^2$. $\Lambda'(f_{\Delta(\tau,\underline{m}),\underline{s}}) $ is a smooth meromorphic section of $|\det\cdot|^{\frac{1-2n}{2}}\rho_{\Delta(\tau,\underline{m}-\underline{1}),\underline{s}}$. Thus, $\mathcal{D}'_{\psi,2n}(E(f_{\Delta(\tau,\underline{m}),\underline{s}}))\circ r'$ is an Eisenstein series attached to $|\det\cdot|^{\frac{1-2n}{2}}\rho_{\Delta(\tau,\underline{m}-\underline{1}),\underline{s}}$.
\end{thm}
\begin{proof}
The proof is similar to that of Theorem \ref{thm 2.2} and we omit it. We just note, as in \eqref{2.16}-\eqref{2.19}, that $\Lambda'$ is the composition of the following operators on $f_{\underline{s}}$: first, we take on each inducing factor $\Delta(\tau,m_i)$, $i=1,2$, the constant term along $V_{(m_i-1)n,n}$. This takes $f_{\underline{s}}$ to $\hat{f}_{\underline{s}}$, which lies in the parabolic induction
\begin{equation}\label{2.34}
\tau|\cdot|^{s_1-\frac{m_1-1}{2}}\times \Delta(\tau,m_1-1)|\cdot|^{s_1+\frac{1}{2}}\times \tau|\cdot|^{s_2-\frac{m_2-1}{2}}\times \Delta(\tau,m_2-1)|\cdot|^{s_2+\frac{1}{2}}
\end{equation}
Then we apply the intertwining operator
\begin{equation}\label{2.35}
M'(\hat{f}_{\underline{s}})(h)=\int\limits_{\mathcal{B}'(\BA)}\hat{f}_{\underline{s}}(w_0\begin{pmatrix}(I_{2n}&b\\&I_{(m-2)n}\end{pmatrix})h)db.
\end{equation}
This operator has a meromorphic continuation to $\BC^2$. It takes the paraolic induction \eqref{2.34} to
\begin{equation}\label{2.36}
\tau|\cdot|^{s_1-\frac{m_1-1}{2}}\times \tau|\cdot|^{s_2-\frac{m_2-1}{2}}\times \Delta(\tau,m_1-1)|\cdot|^{s_1+\frac{1}{2}}\times \Delta(\tau,m_2-1)|\cdot|^{s_2+\frac{1}{2}}
\end{equation}
$$
\cong \tau|\cdot|^{s_1-\frac{m_1-1}{2}}\times \tau|\cdot|^{s_2-\frac{m_2-1}{2}}\times\rho_{\Delta(\tau,\underline{m}-\underline{1}),\underline{s}+\frac{1}{2}\cdot \underline{1}}.
$$
The third operator, applied to \eqref{2.36} is
\begin{equation}\label{2.37}
\mathcal{W}_{\psi_{Z_{2n}}}\otimes id_{\rho_{\Delta(\tau,\underline{m}-\underline{1}),\underline{s}+\frac{1}{2}\cdot \underline{1}}}.
\end{equation}
Thus, 
\begin{equation}\label{2.38}
\Lambda'(f_{\underline{s}})=(\mathcal{W}_{\psi_{Z_{2n}}}\otimes id_{\rho_{\Delta(\tau,\underline{m}-\underline{1}),\underline{s}+\frac{1}{2}\cdot \underline{1}}})\circ M'(\hat{f}_{\underline{s}})\Big|_{r'(\GL_{(m-2)n}(\BA))}
\end{equation}
is a smooth meromorphic section of $|\det\cdot|^{\frac{1-2n}{2}}\rho_{\Delta(\tau,\underline{m}-\underline{1}),\underline{s}}$.

\end{proof}

As in Theorem \ref{thm 2.3}, we have

\begin{thm}\label{thm 2.7}
	
The section $\Lambda'(f_{\Delta(\tau,\underline{m}),\underline{s}})$ is holomorphic at $\underline{s}$, such that 
\begin{equation}\label{2.39}	
Re(s_1-s_2)\geq\frac{m_1-m_2}{2}.
\end{equation}
	
\end{thm}
\begin{proof}
	
The proof is similar to that of Theorem \ref{thm 2.3}. We get that  $\Lambda'(f_{\Delta(\tau,\underline{m}),\underline{s}})(I_{(m-2)n})$ is holomorphic when $\frac{L_S(\tau\times\hat{\tau},s_1-s_2+\frac{m_2-m_1+2}{2})}{L(\tau\times\hat{\tau},s_1-s_2+\frac{m_1+m_2}{2})}$ is holomorphic, which is true when $Re(s_1-s_2)\geq \frac{m_1-m_2}{2}$.
\end{proof}

Finally, we have the analog of Prop. \ref{prop 2.4}, where the local sections $\Lambda' _v(f_{\underline{s},v})$ are defined analogously to \eqref{2.25}.  

\begin{prop}\label{prop 2.8}
Fix a place $v$ and fix $\underline{s}^0\in \BC^2$ satisfying \eqref{2.39}. Consider the local sections $\Lambda' _v(f_{\underline{s},v})$ at the point $\underline{s}^0$. View these as maps to $\rho_{\Delta(\tau_v,\underline{m}-\underline{1}),\underline{s}^0+\frac{1}{2}\cdot \underline{1}}$. Then, when $v$ is finite, these are surjective maps. When $v$ is Archimedean, their image is dense (in the Frechet topology).
\end{prop}

Since $\Delta(\tau,m_1)|\cdot|^{s_1}\times \Delta(\tau,m_2)|\cdot|^{s_2}$ has a central character, it is enough to consider the corresponding Eisenstein series in one complex variable only. We will take $s_1=s$, $s_2=-s$. We will change notation and denote the parabolic induction above by $\rho_{\Delta(\tau,\underline{m}),s}$, and similarly for sections. Our next goal is to determine the poles of the corresponding Eisenstein series, $E(f_{\Delta(\tau,\underline{m}),s})$ when $Re(s)\geq 0$, show that they are simple, and for each such pole and the corresponding residue representation $\pi$, find the set $\mathcal{O}(\pi)$.

\section{The pole at $s=\frac{m_1+m_2}{4}$}

In this section, we prove

\begin{thm}\label{thm 3.1}
The Eisenstein series $E(f_{\Delta(\tau,\underline{m}),s})$ has a simple pole at
$s=\frac{m_1+m_2}{4}$ (as the section varies). It is its right-most pole.
\end{thm}

We break the proof into three steps.

\subsection{The constant term of $E(f_{\Delta(\tau,\underline{m}),s})$ along $V_{n^m}$}
Our Eisenstein series is concentrated on the parabolic subgroup $P_{n^m}$. We will show that the constan term of $E(f_{\Delta(\tau,\underline{m}),s})$ along $V_{n^m}$, $E^{V_{n^m}}(f_{\Delta(\tau,\underline{m}),s})$, has a simple pole at
$s=\frac{m_1+m_2}{4}$, and that it is its right-most pole. Denote, for short $f_s=f_{\Delta(\tau,\underline{m}),s}$. We have, for $Re(s)\gg 0$,\\
\\
$E^{V_{n^m}}(f_s)(g)$=
\begin{equation}\label{3.1}
=\sum\limits_{w\in P_{n\cdot\underline{m}}(F)\backslash \GL_{mn}(F)/P_{n^m}(F)}\int\limits_{[V_{n^m}]}\sum\limits_{\gamma\in P_{n^m}(F)\cap w^{-1}P_{n\cdot\underline{m}}(F)w\backslash P_{n^m}(F)}f_s(w\gamma vg)dv.
\end{equation}

As in the beginning of the proof of Theorem \ref{thm 2.2}, it is enough to take in \eqref{3.1} only permutation matrices $w$ which permute the blocks of $M_{n^m}$. These elements are parametrized by pairs of tuples $(\underline{a},\underline{b})$ of non-negative integers $\underline{a}=(a_1,...,a_r)$, $\underline{b}=(b_1,...,b_r)$, $r\geq 1$, such that $a_1+\cdots+a_r=m_1$, $b_1+\cdots+b_r=m_2$, and, for each $1\leq i\leq r$, $a_i+b_i\geq 1$. The corresponding Weyl element is
\begin{equation}\label{3.2}
w_{n\cdot(\underline{a},\underline{b})}=\begin{pmatrix}w_1^1&0&\cdots&0\\0&w_2^1&\cdots&0\\\\&&\ddots\\0&0&\cdots&w_r^1\\w_1^2&0&\cdots&0\\0&w_2^2&\cdots&0\\&&\ddots\\0&0&\cdots&w_r^2\end{pmatrix},
\end{equation}
$$
w_i^1=(I_{na_i},0)\in M_{na_i\times n(a_i+b_i)}(F),\ \ \ w_i^2=(0,I_{nb_i})\in M_{nb_i\times n(a_i+b_i)}(F),\ 1\leq i\leq r.
$$
We require that $r$ is minimal. For example, this implies that $b_1\neq 0$, since, otherwise, we may join $w^1_1$ and $w^1_2$ to $(I_{n(a_1+a_2)},0)$, and $r$ can be replaced by $r-1$. Similarly, if $a_1=0$, then $a_2\geq 1$. The elements of $P_{n^m}\cap w_{n\cdot(\underline{a},\underline{b})}^{-1}P_{n\cdot\underline{m}}w_{n\cdot(\underline{a},\underline{b})}$ have the form
\begin{equation}\label{3.3}
h=\begin{pmatrix}X_{1,1}&X_{1,2}&\cdots&X_{1,r}\\0&X_{2,2}&\cdots&X_{2,r}\\0&0&\ddots\\0&0&\cdots&X_{r,r}\end{pmatrix},
\end{equation}
where, for $1\leq i\leq j\leq r$,
$$
X_{i,j}=\begin{pmatrix}X^1_{i,j}&X^2_{i,j}\\0&X^4_{i,j}\end{pmatrix}, X^1_{i,j}\in M_{na_i\times na_j},\ X^2_{i,j}\in M_{na_i\times nb_j},\ X^4_{i,j}\in M_{nb_i\times nb_j},
$$
and, for $1\leq i\leq r$,
$$
X^1_{i,i}\in P_{n^{a_i}},\ \ X^4_{i,i}\in P_{n^{b_i}}.
$$
For $h$ as in \eqref{3.3}, we have
\begin{equation}\label{3.4}
w_{n\cdot(\underline{a},\underline{b})}hw_{n\cdot(\underline{a},\underline{b})}^{-1}=\begin{pmatrix}X^1&X^2\\0&X^4\end{pmatrix},
\end{equation}
where, for $e=1,2,4$,
$$
X^e=\begin{pmatrix}X_{1,1}^e&X_{1,2}^e&\cdots&X_{1,r}^e\\0&X_{2,2}^e&\cdots&X_{2,r}^e\\&&\ddots\\0&0&\cdots&X_{r,r}^e\end{pmatrix}.
$$
From \eqref{3.3}, we see that 
$$
P_{n^m}\cap w_{n\cdot(\underline{a},\underline{b})}^{-1}P_{n\cdot\underline{m}}w_{n\cdot(\underline{a},\underline{b})}\backslash P_{n^m}=V_{n^m}\cap w_{n\cdot(\underline{a},\underline{b})}^{-1}P_{n\cdot\underline{m}}w_{n\cdot(\underline{a},\underline{b})}\backslash V_{n^m}(F). 
$$
We take the following subgroup $V_{n^m}^{(\underline{a},\underline{b})}$ as a set of representatives for the last quotient. Its elements have the form
\begin{equation}\label{3.5}
\begin{pmatrix}I_{n(a_1+b_1)}&Y_{1,2}&\cdots&Y_{1,r}\\&I_{n(a_2+b_2)}&\cdots&Y_{2,r}\\\\&&\ddots\\&&&I_{n(a_r+b_r)}\end{pmatrix},
\end{equation}
where, for $1\leq i<j\leq r$,
$$
Y_{i,j}=\begin{pmatrix}0&0\\y_{i,j}&0\end{pmatrix},\ \ y_{i,j}\in M_{nb_i\times na_j}.
$$
Using \eqref{3.4}, \eqref{3.5}, we see that \eqref{3.1} becomes
\begin{equation}\label{3.6}
\sum\limits_{(\underline{a},\underline{b})}\int\limits_{V_{n^m}^{(\underline{a},\underline{b})}(\BA)}f_s^{V_{n^{m_1}}\times V_{n^{m_2}}}(w_{n\cdot(\underline{a},\underline{b})}vg)dv,
\end{equation}
where
\begin{equation}\label{3.7}
f_s^{V_{n^{m_1}}\times V_{n^{m_2}}}(x)=\int\limits_{[V_{n^{m_1}}]\times [V_{n^{m_2}}]}f_s(\diag(v_1,v_2)x)d(v_1,v_2).
\end{equation}
Let us realize the elements of $\Delta(\tau,m_1)^{V_{n^{m_1}}}$ as the residues at $\underline{s}^0=(\frac{m_1-1}{2},...,\frac{1-m_1}{2})$ of the intertwining operator
$$
M_{\underline{s}}:\tau|\cdot|^{s_1}\times\cdots\times\tau|\cdot|^{s_{m_1}}\rightarrow \tau|\cdot|^{s_{m_1}}\times\cdots\times\tau|\cdot|^{s_1},
$$
corresponding to the Weyl element $w_1=\begin{pmatrix}&&&I_n\\&&\cdot\\&\cdot\\I_n\end{pmatrix}\in \GL_{m_1n}(F)$.
Similarly with $\Delta(\tau,m_2)^{V_{n^{m_2}}}$, replacing $\underline{s}$ with $\underline{\zeta}$, $m_1$ by $m_2$; $\underline{\zeta}^0=(\frac{m_2-1}{2},...,\frac{1-m_2}{2})$, and $w_2$ denoting the Weyl element corresponding to $M_{\underline{\zeta}}$. Write \eqref{3.7} as follows. Let $f_{\underline{s}+s\cdot\underline{1},\underline{\zeta}-s\cdot\underline{1}}$ be a smooth, holomorphic section of
$$
\tau|\cdot|^{s_1+s}\times\cdots\times\tau|\cdot|^{s_{m_1}+s}\times \tau|\cdot|^{\zeta_1-s}\times\cdots\times\tau|\cdot|^{\zeta_{m_2}-s}.
$$
Then we realize \eqref{3.7}, for an appropriate $f_{\underline{s}+s\cdot\underline{1},\underline{\zeta}-s\cdot\underline{1}}$,  as $Res_{\underline{s}=\underline{s}^0,\underline{\zeta}=\underline{\zeta}^0}(M_{\underline{s}}\otimes M_{\underline{\zeta}})\circ f_{\underline{s}+s\cdot\underline{1},\underline{\zeta}-s\cdot\underline{1}}$. Now, we view the term in \eqref{3.6}, corresponding to $(\underline{a},\underline{b})$, as
\begin{equation}\label{3.8}
Res_{\underline{s}=\underline{s}^0,\underline{\zeta}=\underline{\zeta}^0}\int\limits_{V_{n^m}^{(\underline{a},\underline{b})}(\BA)}(M_{\underline{s}}\otimes M_{\underline{\zeta}})( f_{\underline{s}+s\cdot\underline{1},\underline{\zeta}-s\cdot\underline{1}}(w_{n(\underline{a},\underline{b})}vg)dv.
\end{equation}
The integral in \eqref{3.8} is the intertwining operator
$$
M_{w^0w_{n\cdot(\underline{a},\underline{b})},\underline{s},\underline{\zeta},s}:\tau|\cdot|^{s_1}\times\cdots\times\tau|\cdot|^{s_m}\rightarrow \tau|\cdot|^{s_{\eta(1)}}\times \tau|\cdot|^{s_{\eta(m)}},
$$
$$
M_{w^0w_{n\cdot(\underline{a},\underline{b})},\underline{s},\underline{\zeta},s}(f_{\underline{s}+s\cdot\underline{1},\underline{\zeta}-s\cdot\underline{1}})(g)=\int\limits_{\tilde{V}_{n^m}^{(\underline{a},\underline{b})}(\BA)}f_{\underline{s}+s\cdot\underline{1},\underline{\zeta}-s\cdot\underline{1}}(w^0w_{n\cdot(\underline{a},\underline{b})}vg)dv,
$$
where $w^0=\diag(w_1,w_2)$, $\tilde{V}_{n^m}^{(\underline{a},\underline{b})}$ is the subgroup of elements of the form
$$
\begin{pmatrix}Y_{1,1}&Y_{1,2}&\cdots&Y_{1,r}\\&Y_{2,2}&\cdots&Y_{2,r}\\&&\ddots\\&&&Y_{r,r}\end{pmatrix},
$$
with the same block division as in \eqref{3.5}, where, for $1\leq i\leq j\leq r$,
$$
Y_{i,j}=\begin{pmatrix}y^1_{i,j}&0\\y^3_{i,j}&y^4_{i,j}\end{pmatrix},\ \ y^1_{i,j}\in M_{na_i\times na_j},\ y^3_{i,j}\in M_{nb_i\times na_j},\ y^4_{i,j}\in M_{nb_i\times nb_j},
$$
and for $1\leq i\leq r$, $y^1_{i,i}\in V_{n^{a_i}},\  y^4_{i,i}\in V_{n^{b_i}}, y^3_{i,i}=0$.
The last integral converges absolutely for $Re(s_i-s_{i+1})\gg 0$, $1\leq i<m_1$, $Re(\zeta_i-\zeta_{i+1})\gg 0$, $1\leq i<m_2$, $2Re((s)+s_{m_1}-\zeta_1)\gg 0$.  
Finally, $\eta$ is the permutation of $1,...,m$ induced by $w^0w_{n\cdot(\underline{a},\underline{b})}$. Summarizing, the constant term \eqref{3.1} is 
\begin{equation}\label{3.9}
E^{V_{n^m}}(f_s)(g)=Res_{\underline{s}=\underline{s}^0,\underline{\zeta}=\underline{\zeta}^0}\sum\limits_{(\underline{a},\underline{b})}M_{w^0w_{n\cdot(\underline{a},\underline{b})},\underline{s},\underline{\zeta},s}(f_{\underline{s}+s\cdot\underline{1},\underline{\zeta}-s\cdot\underline{1}})(g),
\end{equation}
for $f_s=Res_{\underline{s}=\underline{s}^0,\underline{\zeta}=\underline{\zeta}^0}(M_{\underline{s}}\otimes M_{\underline{\zeta}})\circ f_{\underline{s}+s\cdot\underline{1},\underline{\zeta}-s\cdot\underline{1}}$.

\subsection{Analytic properties of ${\bf Res_{\underline{s}=\underline{s}^0,\underline{\zeta}=\underline{\zeta}^0}M_{w^0w_{n\cdot(\underline{a},\underline{b})},\underline{s},\underline{\zeta},s}(f_{\underline{s}+s\cdot\underline{1},\underline{\zeta}-s\cdot\underline{1}})}$}

Examining the inversion set of the permutation $\eta$ (the set of $(i,j)$, $1\leq i<j\leq m$, such that $\eta(i)>\eta(j)$), a calculation, using the Gindikin-Karpelevich formula, shows that
\begin{equation}\label{3.10}
Res_{\underline{s}=\underline{s}^0,\underline{\zeta}=\underline{\zeta}^0}M_{w^0w_{n\cdot(\underline{a},\underline{b})},\underline{s},\underline{\zeta},s}(f_{\underline{s}+s\cdot\underline{1},\underline{\zeta}-s\cdot\underline{1}})(I_{mn})=c_{\underline{a},\underline{b}}(s)N_s(f_{\underline{s}^0+s\cdot\underline{1},\underline{\zeta}^0-s\cdot\underline{1}})(I_{mn}),
\end{equation}
$$
c_{\underline{a},\underline{b}}(s)=\prod\limits_{i=1}^{r-1}\prod\limits_{j=1}^{b_i}\frac{L(\tau\times\hat{\tau},2s-\frac{m}{2}+b_1+\cdots+b_{i-1}+j)}{L(\tau\times\hat{\tau},2s-\frac{m}{2}+a_{i+1}+\cdots+a_r+b_1+\cdots+b_{i-1}+j)}.
$$
Here, $N_s=N_{w^0w_{n\cdot(\underline{a},\underline{b})},\underline{s}^0+s\cdot\underline{1},\underline{\zeta}^0-s\cdot\underline{1}}$ denotes the normalized intertwining operator corresponding to $M_{w^0w_{n\cdot(\underline{a},\underline{b})},\underline{s}+s\cdot\underline{1},\underline{\zeta}-s\cdot\underline{1}}$, at $\underline{s}=\underline{s}^0$, $\underline{\zeta}=\underline{\zeta}^0$. By \cite{MW89}, I.10, it is holomorphic at $\underline{s}=\underline{s}^0$, $\underline{\zeta}=\underline{\zeta}^0$, $2Re(s)+\frac{1-m_1}{2}-\frac{m_2-1}{2}\geq 0$, i.e. $Re(s)\geq \frac{m-2}{4}$, and in particular it is holomorphic at $s=\frac{m}{4}=\frac{m_1+m_2}{4}$. Let us examine the factors of $c_{\underline{a},\underline{b}}(s)$ at $s=\frac{m}{4}$. Assume that $r\geq 2$. Note that in the factor corresponding to $(i,j)$, when $b_i\geq 1$, for $s=\frac{m}{4}$,
$$
2s-\frac{m}{2}+b_1+\cdots+b_{i-1}+j=b_1+\cdots+b_{i-1}+j\geq j\geq 1,
$$
$$
2s-\frac{m}{2}+a_{i+1}+\cdots+a_r+b_1+\cdots+b_{i-1}+j=a_{i+1}+\cdots+a_r+b_1+\cdots+b_{i-1}+j\geq j\geq 1.
$$
We conclude that at $s=\frac{m}{4}$, $c_{\underline{a},\underline{b}}(s)$ has at most a simple pole. This shows that \eqref{3.9} has at most a simple pole at $s=\frac{m}{4}$ This also shows that, for $Re(s)>\frac{m}{4}$, \eqref{3.9} is holomorphic.

\subsection{Proof that $s=\frac{m}{4}$ is a pole of $E(f_s)$, for a certain section}
We will choose a specific section $f_s$, such that $E(f_s)$ has a pole at $s=\frac{m}{4}$. Fix a finite place of $F$, $v_0$. Choose $f_s$, such that its support at $v_0$ is in the open cell $P_{n\cdot\underline{m}}(F_{v_0})\begin{pmatrix}&I_{m_2n}\\I_{m_1n}\end{pmatrix}V_{n^m}(F_{v_0})$, compact modulo $P_{n\cdot\underline{m}}(F_{v_0})$ from the left. This implies that, for $g\in \GL_{mn}(\BA)$, such that $g_{v_0}$ does not lie in this open cell, $f_s(g)=0$. For this section, all terms of \eqref{3.9} are zero, except the term with $r=2,\ a_1=0,\ a_2=m_1,\ b_1=m_2,\ b_2=0$. The corresponding function in \eqref{3.10} is
$$
c_{\underline{a},\underline{b}}(s)=\prod\limits_{j=1}^{m_2}\frac{L(\tau\times\hat{\tau},2s-\frac{m}{2}+j)}{L(\tau\times\hat{\tau},2s-\frac{m}{2}+m_2+j)},
$$
which clearly has a simple pole at $s=\frac{m}{4}$. Finally, it is clear that we can choose $f_s$, as above, such that in \eqref{3.10}, $N_s(f_{\underline{s}^0+s\cdot\underline{1},\underline{\zeta}^0-s\cdot\underline{1}})(I_{mn})$ is nonzero at $s=\frac{m}{4}$. This completes the proof of Theorem \ref{thm 3.1}.

\section{The positive poles of $E(f_{\Delta(\tau,\underline{m}),\underline{s}})$}

In this section, we find the list of poles with positive part of $E(f_{\Delta(\tau,\underline{m}),s})$. We first find a list of such poles as a corollary of Theorem \ref{thm 3.1}. Then we prove that these poles are simple and that there are no other poles. 

\subsection{Poles of $E(f_{\Delta(\tau,\underline{m}),s})$ derived from Theorem \ref{thm 3.1}}

\begin{thm}\label{thm 4.1}
The Eisenstein series $E(f_{\Delta(\tau,\underline{m}),s})$ has poles at \begin{equation}\label{4.1}s=\frac{m}{4}-\frac{i}{2},\ \  0\leq i\leq min\{m_1,m_2\}-1.
\end{equation}
\end{thm}
\begin{proof}
The proof is similar to the proof of Theorem 2.2, part I, in \cite{GS22}. The idea is to apply repeatedly the descent to our Eisenstein series and use Theorem \ref{thm 3.1} at each step. Assume that $m_2\leq m_1$. We prove the theorem by induction on $m=m_1+m_2$. If $m=2$, then $m_1=m_2=1$ and this is a special case of Theorem \ref{thm 3.1}, so there is nothing to prove. Assume that $m>2$ and that the statement of the theorem is true for smaller values of $m$. If $m_2=1$, then the theorem follows from Theorem \ref{thm 3.1}. Assume that $m_2>1$. By Theorem \ref{thm 2.2}, for $g\in \GL_{(m-2)n}(\BA)$, 
\begin{equation}\label{4.2}
\mathcal{D}_{\psi,2n}(E(f_{\Delta(\tau,\underline{m}),s}))(r(g))=|\det\cdot|^{\frac{2n-1}{2}}E(\varphi_{\Delta(\tau,m_1-1)|\cdot|^s\times\Delta(\tau,m_2-1)|\cdot|^{-s}})(g),
\end{equation}
where $E(\varphi_{\Delta(\tau,m_1-1)|\cdot|^s\times\Delta(\tau,m_2-1)|\cdot|^{-s}})$ is an Eisenstein series on $\GL_{(m-2)n}(\BA)$, corresponding to the section of 
$\Delta(\tau,m_1-1)|\cdot|^s\times\Delta(\tau,m_2-1)|\cdot|^{-s}$,
$$
\varphi_{\Delta(\tau,m_1-1)|\cdot|^s\times\Delta(\tau,m_2-1)|\cdot|^{-s}}=\Lambda(f_{\Delta(\tau,\underline{m}),s}).
$$
Let $s(i)=\frac{m}{4}-\frac{i}{2}$, $1\leq i\leq m_2-1$. Then $s(i)=\frac{(m_1-1)+(m_2-1)}{4}-\frac{i-1}{2}$, $0\leq i-1\leq (m_2-1)-1$. By Theorem \ref{thm 2.3}, since we assume that $m_2\leq m_1$,\\ $\varphi_{\Delta(\tau,m_1-1)|\cdot|^s\times\Delta(\tau,m_2-1)|\cdot|^{-s}}$ is holomorphic at $s(i)$. By Prop. \ref{prop 2.4} and the induction hypothesis, $\mathcal{D}_{\psi,2n}(E(f_{\Delta(\tau,\underline{m}),s}))(I_{(m-2)n})$ has a pole at $s(i)$, for some section $f_{\Delta(\tau,\underline{m}),s}$. This implies that $E(f_{\Delta(\tau,\underline{m}),s})$ has a pole at $s(i)$. For $i=0$, this is Theorem \ref{thm 3.1}.

We note again that when $m_1<m_2$, we use Prop. \ref{prop 2.5}, Theorem \ref{thm 2.6}, Theorem \ref{thm 2.7}, and Prop. \ref{prop 2.8}. We get the analog of \eqref{4.2},
\begin{equation}\label{4.3}
\mathcal{D'}_{\psi,2n}(E(f_{\Delta(\tau,\underline{m}),s}))(r'(g))=|\det\cdot|^{\frac{1-2n}{2}}E(\varphi'_{\Delta(\tau,m_1-1)|\cdot|^s\times\Delta(\tau,m_2-1)|\cdot|^{-s}})(g),
\end{equation}
$$
\varphi'_{\Delta(\tau,m_1-1)|\cdot|^s\times\Delta(\tau,m_2-1)|\cdot|^{-s}}=\Lambda'(f_{\Delta(\tau,\underline{m}),s}).
$$
We get in exactly the same way that the points $s=\frac{m}{4}-\frac{i}{2}$, $0\leq i\leq m_1-1$ are all poles of $E(f_{\Delta(\tau,\underline{m}),s})$.	
	
\end{proof}

We remark that we don't know yet that the poles \eqref{4.1} of our Eisenstein series are simple, nor that these are all the poles with positive real part.

\subsection{Proof of simplicity of poles and that there are no other poles with positive real part}

\begin{thm}\label{thm 4.2}
The set \eqref{4.1} is the full set of poles with positive real part of the Eisenstein series $E(f_{\Delta(\tau,\underline{m}),s})$. They are all simple.
\end{thm}
\begin{proof}
The proof is by induction on $m=m_1+m_2$. If $m=2$, then our Eisenstein series corresponds to $\tau|\cdot|^s\times\tau|\cdot|^{-s}$. In this case, the theorem is well known. In fact, the formula \eqref{3.9} becomes very simple. In the same notation,
$$
E^{V_{n^2}}(f_s)=f_s+M_s(f_s),
$$
where $M_s$ is the intertwining operator with respect to $\begin{pmatrix}&I_n\\I_n\end{pmatrix}$. The corresponding term $c(s)$ in \eqref{3.10} is $\frac{L(\tau\times\hat{\tau},2s)}{L(\tau\times\hat{\tau},2s+1)}$. The poles with positive real part of this ratio are those of $L(\tau\times\hat{\tau},2s)$, i.e. $s=\frac{1}{2}$.

Assume, by induction, that the statement of the theorem is corret, for all $m'_1+m'_2<m$. Let $s_0$ be a pole of $E(f_{\Delta(\tau,\underline{m}),s})$, with $Re(s_0)>0$. Consider the Laurent expansion of $E(f_{\Delta(\tau,\underline{m}),s})$ around $s_0$. Assume that $s_0$ is a pole of order $e$. Denote by $\pi=\pi(\tau,\underline{m},s_0)$ the automorphic representation of $\GL_{mn}(\BA)$ generated by the leading terms $A_{-e}(f_{\Delta(\tau,\underline{m}),s_0})$ of the last Laurent expansion, as the section $f_{\Delta(\tau,\underline{m}),s}$ varies. Assume that $\mathcal{D}_{\psi,2n}(\pi)\neq 0$. By Theorems \ref{thm 2.2}, \ref{thm 2.3}, \ref{thm 2.6}, \ref{thm 2.7}, and Propositions \ref{prop 2.4}, \ref{2.8}, $s_0$ is a pole of the descent (via $\mathcal{D}_{\psi,2n}$, or $\mathcal{D}'_{\psi,2n}$) of $E(f_{\Delta(\tau,\underline{m}),s})$, which is, up to a power of $|\det\cdot|$, an Eisenstein series corresponding to $\Delta(\tau,m_1-1)|\cdot|^s\times \Delta(\tau,m_2-1)|\cdot|^{-s}$. By induction, we have that $e=1$, and hence $s_0$ is a simple pole of $E(f_{\Delta(\tau,\underline{m}),s})$. Moreover, there is $0\leq i\leq min\{m_1-1,m_2-1\}-1$, such that 
$$
s_0=\frac{m-2}{4}-\frac{i}{2}=\frac{m}{4}-\frac{i+1}{2}.
$$
Note that $1\leq i+1\leq min\{m_1.m_2\}-1$. Thus, it remains to prove the theorem when $\mathcal{D}_{\psi,2n}(\pi)=0$.

\begin{prop}\label{prop 4.3}
Keeping the notation above, assume that $\mathcal{D}_{\psi,2n}(\pi)=0$. Then
\begin{equation}\label{4.4}
\mathcal{O}(\pi)=(n^m).
\end{equation}
\end{prop}

\begin{proof}
Since $\pi$ is concentrated on $P_{n^m}$, and $\tau$ is generic, $\pi$ admits the nontrivial Fourier coefficient \eqref{1.48}, $\mathcal{Z}_{\psi, n,m}$. By Theorem \ref{thm 1.6}, the Fourier coefficient $\mathcal{F}_{\psi_{(n^m)}}$ is nontrivial on $\pi$. Let $\underline{p}$ be a partition of $mn$, attached to an element of $\mathcal{O}(\pi)$. By \eqref{2.2}, $\underline{p}$ is bounded by $((2n)^{m_2},n^{m_1-m_2})$.	The partition $\underline{p}$ cannot be of the form $((2n)^r,...)$ with $1\leq r\leq m_2$. Indeed, in such a case, it follows, exactly as in \cite{GRS03}, Lemma 2.6, that $\mathcal{F}_{\psi_{(2n,1^{(m-2)n})}}$ is nontrivial on $\pi$. By Prop. \ref{prop 1.4} and Prop. \ref{prop 2.1}, we get that $\mathcal{D}_{\psi,2n}(\pi)\neq 0$, contradicting our assumption. Thus, $\underline{p}$ has the form $(n',...)$, $n'<2n$. Note that $n'\geq n$, since otherwise $\underline{p}<(n^m)$, which is impossible, since $\mathcal{F}_{\psi_{(n^m)}}$ is nontrivial on $\pi$. Assume that $n<n'<2n$. Then, as we just argued, as in \cite{GRS03}, Lemma 2.6, it follows that $\mathcal{F}_{\psi_{(n',1^{mn-n'})}}$ is nontrivial on $\pi$. By Prop. \ref{prop 1.4}, $\mathcal{D}^0_{\psi,n'}(\pi)\neq 0$. The same proof as that of Prop. \ref{prop 2.1} (with $n'$ in place of $kn$) shows that $\mathcal{D}_{\psi,n'}(\pi)\neq 0$. This implies that the constant term along $V_{mn-n',n'}$ is nontrivial on $\pi$. Since $\tau$ is a cuspidal representation of $\GL_n(\BA)$, $n'$ must be a multiple of $n$. This is impossible, since $n<n'<2n$. This shows that $n'=n$. Since $\mathcal{F}_{\psi_{(n^m)}}$ is nontrivial on $\pi$, we conclude that $\underline{p}=(n^m)$.	
	
\end{proof}

The last proposition allows us to use Theorem \ref{thm 1.7}. We will use it for $a=\diag(c_1,...,c_m)$, $c_1,...,c_m\in \BA^*$, $c_1\cdot\dots\cdot c_m=1$, in \eqref{1.61}. Re-denote $c=\diag(c_1,...,c_m)$. For $\varphi\in \pi$,
\begin{equation}\label{4.5}
\mathcal{F}_{\psi_{n^m}}(\varphi)(\delta_n(c)g)=\mathcal{F}_{\psi_{n^m}}(\varphi)(g).
\end{equation}
We may take $\varphi=A_{-e}(f_{\Delta(\tau,\underline{m}),s_0})$. Denote by $B_{-e,\psi_{n^m}}(f_{\Delta(\tau,\underline{m}),s_0})$ the coefficient of $(s-s_0)^{-e}$ in the Laurent expansion around $s_0$ of $\mathcal{F}_{\psi_{n^m}}(E(f_{\Delta(\tau,\underline{m}),s}))$. We denote similarly by $C_{-e,\psi_{n^m}}(f_{\Delta(\tau,\underline{m}),s_0})$ the coefficient of $(s-s_0)^{-e}$ in the Laurent expansion of $\mathcal{Z}_{\psi,n,m}(E(f_{\Delta(\tau,\underline{m}),s}))$ around $s_0$. Then
\begin{equation}\label{4.6}
\mathcal{F}_{\psi_{n^m}}(\varphi)(g)=B_{-e,\psi_{n^m}}(f_{\Delta(\tau,\underline{m}),s_0})(g).
\end{equation}
Let us apply Theorem \ref{thm 1.6}. More precisely, let us use \eqref{1.60}. Using \eqref{4.6}, we get, in the notation of \eqref{1.60},
\begin{equation}\label{4.7}
B_{-e,\psi_{n^m}}(f_{\Delta(\tau,\underline{m}),s_0})(g)=
\int\limits_{Y(\BA)}C_{-e,\psi_{n^m}}(f_{\Delta(\tau,\underline{m}),s_0})(yw_0g)dy.
\end{equation}
Let us use \eqref{3.9}, with its notations, to explicate the r.h.s. of \eqref{4.7}. Denote by $D_{-e,\psi_{n^m},\underline{a},\underline{b}}(f_{\underline{s}^0+s_0\cdot\underline{1},\underline{\zeta}^0-s_0\cdot\underline{1}})$ the coefficient of $(s-s_0)^{-e}$ in the Laurent expansion around $s_0$ of $Res_{\underline{s}=\underline{s}^0,\underline{\zeta}=\underline{\zeta}^0}M^{\psi_{Z^m_n}}_{w^0w_{n\cdot(\underline{a},\underline{b})},\underline{s},\underline{\zeta},s}(f_{\underline{s}+s\cdot\underline{1},\underline{\zeta}-s\cdot\underline{1}})$, where\\
\\
$M^{\psi_{Z^m_n}}_{w^0w_{n\cdot(\underline{a},\underline{b})},\underline{s},\underline{\zeta},s}(f_{\underline{s}+s\cdot\underline{1},\underline{\zeta}-s\cdot\underline{1}})(g)=$

$$
=\int\limits_{[Z_n]^m}M_{w^0w_{n\cdot(\underline{a},\underline{b})},\underline{s},\underline{\zeta},s}(f_{\underline{s}+s\cdot\underline{1},\underline{\zeta}-s\cdot\underline{1}})(\diag(z_1,...,z_m)g)\psi^{-1}_{Z_n}(z_1\cdots z_m)dz.
$$
Then by \eqref{4.5}, we have that
\begin{equation}\label{4.8}
\sum\limits_{(\underline{a},\underline{b})}\int\limits_{Y(\BA)}D_{-e,\psi_{n^m},\underline{a},\underline{b}}(f_{\underline{s}^0+s_0\cdot\underline{1},\underline{\zeta}^0-s_0\cdot\underline{1}})(yw_0\delta_n(c)g)dy
\end{equation}
is independent of $c_1,...,c_m\in \BA^*$, such that $c_1\cdots c_m=1$. 
Note that
$$
w_0\delta_n(c)w_0^{-1}=\diag(c_1I_n,...,c_mI_n):=d_n(c).
$$
Then \eqref{4.8} is equal to
\begin{equation}\label{4.9}
\delta_{B_m}^{-\frac{n(n-1)}{2}}(c)\sum\limits_{(\underline{a},\underline{b})}\int\limits_{Y(\BA)}D_{-e,\psi_{n^m},\underline{a},\underline{b}}(f_{\underline{s}^0+s_0\cdot\underline{1},\underline{\zeta}^0-s_0\cdot\underline{1}})(d_n(c)yw_0g)dy
.
\end{equation}
We have
\begin{equation}\label{4.10}
w^0w_{n\cdot(\underline{a},\underline{b})}d_n(c)w^{-1}_{n\cdot(\underline{a},\underline{b})}(w^0)^{-1}=\diag(\eta_1(c),...,\eta_r(c),\mu_1(c),...,\mu_r(c)),
\end{equation}
$$
\eta_i(c)=\diag(c_{\sum\limits_{j\leq i-1}(a_j+b_j)+a_i}I_n,...,c_{\sum\limits_{j\leq i-1}(a_j+b_j)+2}I_n,c_{\sum\limits_{j\leq i-1}(a_j+b_j)+1}I_n),\ \ 1\leq i\leq r,
$$
$$
\mu_i(c)=\diag(c_{\sum\limits_{j\leq i}(a_j+b_j)I_n},...,c_{\sum\limits_{j\leq i-1}(a_j+b_j)+a_i+2}I_n,c_{\sum\limits_{j\leq i-1}(a_j+b_j)+a_i+1}I_n),\ \ 1\leq i\leq r.
$$
Let $\chi$ denote the following character of the center of $M_{n^m}(\BA)$.\\  For $d_n(t)=\diag(t_1I_n,...,t_mI_n)$, $t_1,...,t_m\in \BA^*$,
$$
\chi(d_n(t))=\prod\limits_{i=1}^{m_1}|t_i|^{n(\frac{2i-1-m_1}{2}+s_0)}\prod\limits_{j=1}^{m_2}|t_{m_1+j}|^{n(\frac{2j-1-m_2}{2}-s_0)}.
$$
Then by \eqref{4.10},
\begin{equation}\label{4.11}
\delta_{B_m}^{-\frac{n(n-1)}{2}}(c)D_{-e,\psi_{n^m},\underline{a},\underline{b}}(f_{\underline{s}^0+s_0\cdot\underline{1},\underline{\zeta}^0-s_0\cdot\underline{1}})(d_n(c)h)=
\end{equation}
$$
=\chi_{\underline{a},\underline{b},s_0}(c)D_{-e,\psi_{n^m},\underline{a},\underline{b}}(f_{\underline{s}^0+s_0\cdot\underline{1},\underline{\zeta}^0-s_0\cdot\underline{1}})(h),
$$
where 
$$
\chi_{\underline{a},\underline{b},s_0}(c)=\delta_{B_m}^{-\frac{n(n-1)}{2}}(c)\delta_{P_{n^m}}^{\frac{1}{2}}(d_n(c))\chi(w^0w_{n\cdot(\underline{a},\underline{b})}d_n(c)w^{-1}_{n\cdot(\underline{a},\underline{b})}(w^0)^{-1})=
$$
$$
=\delta_{B_m}^{\frac{n}{2}}(c)\chi(w^0w_{n\cdot(\underline{a},\underline{b})}d_n(c)w^{-1}_{n\cdot(\underline{a},\underline{b})}(w^0)^{-1}).
$$
Thus,
\begin{equation}\label{4.12}
\sum\limits_{(\underline{a},\underline{b})}\chi_{\underline{a},\underline{b},s_0}(c)\int\limits_{Y(\BA)}D_{-e,\psi_{n^m},\underline{a},\underline{b}}(f_{\underline{s}^0+s_0\cdot\underline{1},\underline{\zeta}^0-s_0\cdot\underline{1}})(yw_0g)dy
\end{equation}
is independent of $c_1,...,c_m\in \BA^*$, $c_1\cdots c_m=1$. It follows that if $\chi_{\underline{a},\underline{b},s_0}$ is nontrivial, then it does not contribute to \eqref{4.12}. Indeed, let $(\underline{\alpha}^i,\underline{\beta}^i)$, $1\leq i\leq k_0$, all the pairs in \eqref{4.12}, such that $\chi_{\underline{\alpha}^i,\underline{\beta}^i,s_0}=\chi_{\underline{a},\underline{b},s_0}$. Then, by linear independence of characters, we must have that
\begin{equation}\label{4.13}
\sum\limits_{i=1}^{k_0}\int\limits_{Y(\BA)}D_{-e,\psi_{n^m},\underline{\alpha}^i,\underline{\beta}^i}(f_{\underline{s}^0+s_0\cdot\underline{1},\underline{\zeta}^0-s_0\cdot\underline{1}})(yw_0g)dy=0.
\end{equation}
Thus, for $(\underline{a},\underline{b})$ to contribute to \eqref{4.12}, it is necessary that $\chi_{\underline{a},\underline{b},s_0}$ be trivial. Note that \eqref{4.12} is nontrivial since, by \eqref{4.5}-\eqref{4.9}, it is equal to $\mathcal{F}_{\psi_{n^m}}(\varphi)$, which is nontrivial on $\pi$. Write
$$
\chi_{\underline{a},\underline{b},s_0}(c)=\prod\limits_{i=1}^{m-1}|c_i|^{n_i(\underline{a},\underline{b},s_0)}.
$$ 
By Theorem \ref{thm 3.1}, we may assume that $Re(s_0)<\frac{m}{4}$. Assume that $a_1\geq 1$, $b_r\geq 1$. Then, looking at \eqref{4.10}, and using that $|c_m|=|c_1\cdots c_{m-1}|^{-1}$, we see that 
\begin{equation}\label{4.14}
n_1(\underline{a},\underline{b},s_0)=n(2s_0+\frac{3m}{2}-(a_1+b_r)). 
\end{equation}
Since $Re(s_0)\geq 0$ and $a_1+b_r\leq m$, $n_1(\underline{a},\underline{b},s_0)\geq \frac{m}{2}$, and hence $\chi_{\underline{a},\underline{b},s_0}$ is nontrivial. Assume that $a_1\geq 1$, $b_r=0$. Then 
\begin{equation}\label{4.15}
n_1(\underline{a},\underline{b},s_0)=n(m_1-(a_1+a_r)+m)\geq mn.
\end{equation}
and hence $\chi_{\underline{a},\underline{b},s_0}$ is nontrivial. Thus, for $\chi_{\underline{a},\underline{b},s_0}$ to be trivial, we must have $a_1=0$. It follows that $b_1\geq 1$. If $b_r\geq 1$, then 
\begin{equation}\label{4.16}
n_1(\underline{a},\underline{b},s_0)=n(m_1-(b_1+b_r)+m)\geq mn. 
\end{equation}
Hence $b_r=0$, and then 
\begin{equation}\label{4.17}
n_1(\underline{a},\underline{b},s_0)=n(\frac{3m}{2}-2s_0-(a_r+b_1)).
\end{equation}
For this to be zero, we must have $2s_0-\frac{m}{2}=m-(a_r+b_1)\geq 0$, and hence $s_0$ is real and $s_0\geq \frac{m}{4}$, contradicting our assumption that $Re(s_0)<\frac{m}{4}$. This concludes the proof of Theorem \ref{thm 4.2}.

\end{proof}

\section{$E(f_{\Delta(\tau,(k,k)),s})$ is zero at $s=0$}
We will use the ideas in the last section to prove
\begin{thm}\label{thm 5.1}
For any smooth holomorphic section $f_{\Delta(\tau,(k,k)),s}$,
$$
E(f_{\Delta(\tau,(k,k)),0})=0.
$$
\end{thm}

\begin{proof}
We know that $E(f_{\Delta(\tau,(k,k)),s})$ is holomorphic at $s=0$. This is a general fact, since $\tau$ is unitary. See \cite{MW95}, IV.1.11(b). We will prove the theorem by induction on $k$. If $k=1$, this follows from a result of Keys and Shahidi, stating that, in this case, $M_0=-id$, where $M_s: \tau|\cdot|^s\times\tau|\cdot|^{-s}\rightarrow \tau|\cdot|^{-s}\times\tau|\cdot|^s$ is the intertwining operator corresponding to $\begin{pmatrix}&I_n\\I_n\end{pmatrix}$. See \cite{KS88}, Prop. 6.3. We have the functional equation (\cite{MW95}, IV.1.10)
$$
E(f_{\Delta(\tau,(1,1),s)})=E(M_s(f_{\Delta(\tau,(1,1),s)})).
$$
We conclude that $E(f_{\Delta(\tau,(1,1),0)})=E(M_0(f_{\Delta(\tau,(1,1),0)}))=-E(f_{\Delta(\tau,(1,1),0)})$ and hence $E(f_{\Delta(\tau,(1,1),0)})=0$.

Let $k\geq 2$. Assume that our Eisenstein series is not identically zero, at $s=0$. Let $\pi$ be the representation of $\GL_{2nk}(\BA)$, acting by right translations in the space  generated by all the $E(f_{\Delta(\tau,(k,k)),0})$. Consider $\mathcal{D}_{\psi,2n}(\pi)\circ r$. By Theorem \ref{thm 2.2}, up to $|\det\cdot|^{\frac{2n-1}{2}}$, this is an Eisenstein series at $s=0$, corresponding to $\Delta(\tau,k-1)|\cdot|^s\times \Delta(\tau,k-1)|\cdot|^{-s}$. By induction, $\mathcal{D}_{\psi,2n}(\pi)=0$. The proof of Prop. \ref{prop 4.3} applies here in exactly the same way, and we conclude that $\mathcal{O}(\pi)=(n^{2k})$. Now, we get \eqref{4.5} with $\varphi= E(f_{\Delta(\tau,(k,k)),0})$.	We use the same notation. For all $c_1,...,c_{2k}\in \BA^*$, such that $c_1\cdots c_{2k}=1$, 
\begin{equation}\label{5.1}
\mathcal{F}_{\psi_{n^{2k}}}(\varphi)(\delta_n(c)g)=\mathcal{F}_{\psi_{n^{2k}}}(\varphi)(g).	
\end{equation}
As in \eqref{4.7}, we get	
\begin{equation}\label{5.2}
\mathcal{F}_{\psi_{n^{2k}}}(E(f_{\Delta(\tau,(k,k)),0}))(g)=
\int\limits_{Y(\BA)}\mathcal{Z}_{\psi,n,2k}(E(f_{\Delta(\tau,(k,k)),0}))(yw_0g)dy.
\end{equation}
As in \eqref{4.8}, denote by $D_{0,\psi_{n^{2k}},\underline{a},\underline{b}}(f_{\underline{s}^0,\underline{\zeta}^0})$ the value at $s=0$ of\\ $Res_{\underline{s}=\underline{s}^0,\underline{\zeta}=\underline{\zeta}^0}M^{\psi_{Z^{2k}_n}}_{w^0w_{n\cdot(\underline{a},\underline{b})},\underline{s},\underline{\zeta},s}(f_{\underline{s}+s\cdot\underline{1},\underline{\zeta}-s\cdot\underline{1}})$.
Then 
\begin{equation}\label{5.3}
\mathcal{F}_{\psi_{n^{2k}}}(\varphi)(\delta_n(c)g)=\sum\limits_{(\underline{a},\underline{b})}\int\limits_{Y(\BA)}D_{0,\psi_{n^{2k}},\underline{a},\underline{b}}(f_{\underline{s}^0,\underline{\zeta}^0})(yw_0\delta_n(c)g)dy
\end{equation}
is independent of $c_1,...,c_{2k}\in \BA^*$, such that $c_1\cdots c_{2k}=1$. As in \eqref{4.12},
\begin{equation}\label{5.4}
\sum\limits_{(\underline{a},\underline{b})}\chi_{\underline{a},\underline{b},0}(c)\int\limits_{Y(\BA)}D_{0,\psi_{n^{2k}},\underline{a},\underline{b}}(f_{\underline{s}^0,\underline{\zeta}^0})(yw_0g)dy
\end{equation}
is independent of $c_1,...,c_{2k}\in \BA^*$, such that $c_1\cdots c_{2k}=1$. All the characters (in $c_1,...,c_{2k-1}$) $\chi_{\underline{a},\underline{b},0}$ are nontrivial. This is clear from \eqref{4.14}-\eqref{4.17}, replacing $s_0$ there by zero. For example, in the case of \eqref{4.17}, $n_1(\underline{a},\underline{b},0)=n(\frac{3m}{2}-(a_r+b_1))\geq \frac{mn}{2}$. As we explained in \eqref{4.13}, it follows, by linear independence of characters, that \eqref{5.4} is zero, and then, from \eqref{5.1}- \eqref{5.3}, $\mathcal{F}_{\psi_{n^{2k}}}$ is zero on $\pi$. This is a contradiction, since $\mathcal{O}(\pi)=(n^{2k})$. This proves the theorem.
\end{proof}

\section{Top Fourier coefficients of $Res_{s=\frac{m}{4}-\frac{i}{2}}F(f_{\Delta(\tau,(m_1,m_2)),s})$}

Let $0\leq i\leq min(m_1,m_2)-1$. Denote by $\pi_i=\pi_i^{(m_1,m_2)}(\tau)$ the representation (by right translations) of $\GL_{mn}(\BA)$ in the space generated by the residues $Res_{s=\frac{m}{4}-\frac{i}{2}}E(f_{\Delta(\tau,(m_1,m_2)),s})$. 
\begin{thm}\label{thm 6.1}	
We have
$$
\mathcal{O}(\pi_i)=((2n)^i,n^{m-2i}).
$$
\end{thm}
\begin{proof}
The proof is similar to the proof of Theorem 3.2 in \cite{GS22}. We start with the case $i=0$. Here, one can see that
$\pi_0=|\det\cdot|^{\frac{m_1-m_2}{4}}\Delta(\tau,m)$. It is known that $\mathcal{O}(\Delta(\tau,m))=(n^m)$. See Prop. 5.3 in \cite{G06}. See also \cite{JL13}. Let $0\leq i\leq min(m_1,m_2)-1$. Assume that $m_2\leq m_1$. We show, by induction on $m$, that $\mathcal{F}_{\psi_{((2n)^i,n^{m-2i})}}$ is nontrivial on $\pi_i$. If $m=2$, then necessarily $i=0$, and we took care of this case. Assume that $m>2$. If $m_2=1$, then $i=0$ and we took care of this case. Assume that $m_2>1$, $1\leq i\leq m_2-1$. By Theorem \ref{thm 4.1}, $\mathcal{D}_{\psi,2n}(\pi_i)=|\det\cdot|^{\frac{2n-1}{2}}\pi_{i-1}^{(m_1-1,m_2-1)}(\tau)$. By induction and as we did in the proof of Prop. \ref{prop 4.3}, it follows, exactly as in \cite{GRS03}, Lemma 2.6, that the Fourier coefficient corresponding to $(2n,(2n)^{i-1},n^{(m-2)-2(i-1)})=((2n)^i,n^{m-2i})$ supports $\pi_i$. Examining, at almost all places $\nu$, the unramified constituent at $\nu$ of $\Delta(\tau,m_1)|\cdot|^{\frac{m}{4}-\frac{i}{2}}\times \Delta(\tau,m_2)|\cdot|^{-\frac{m}{4}+\frac{i}{2}}$, we find that it coincides with the unramified constituent at $\nu$ of $|\det\cdot|^{\frac{m_1-m_2}{4}}(\Delta(\tau,m-i)\times \Delta(\tau,i)|\cdot|^{\frac{m_2-m_1}{2}})$. This bounds the elements of $\mathcal{O}(\pi_i)$ by $((2n)^i,n^{m-2i})$. The argument is standard and uses one place $\nu$ where $\tau$ is unramified. See, for example, \cite{LX20}, Prop. 5.1. This proves the theorem. If $m_1<m_2$, the proof is similar, using \eqref{4.3}.

\end{proof}

\end{document}